\documentclass{article}

\usepackage{cancel}
\usepackage{proof}
\usepackage{amsthm,amsmath,amssymb}
\usepackage{macros}
\usepackage[all]{xy}
\usepackage{enumerate}
\usepackage{hyperref}

\renewcommand{\SS}{{\mathcal S}} 

\title{Exploring the Landscape of Relational Syllogistic Logics}
\author{Alex Kruckman and Lawrence S.\ Moss
}

\begin{document}

\maketitle

\begin{abstract}
This paper explores relational syllogistic logics, a family of logical systems related to reasoning about relations in 
extensions of the classical syllogistic.
These are all decidable logical systems.
We prove completeness theorems and complexity results for a natural subfamily of relational syllogistic logics, parametrized by constructors for terms and for sentences.
\end{abstract}

\tableofcontents

\section{Introduction}

This paper explores several fragments of \emph{relational syllogistic logic} and aims to 
provide completeness and complexity results for them.   These are among the simplest of all logical systems.
To set notation and terminology by example, let us consider the absolutely simplest logical system $\langall$, the one for sentences
``all $p$ are $q$'' introduced in~\cite{logic:moss08}. The syntax begins with a set $\bP$, called the \emph{nouns}. 
Then the sentences of $\langall$ are simply the expressions 
$(\all{p}{q})$, where $p$ and $q$ are nouns.   The semantics is equally straightforward.
A \emph{model} $\Model$ is a set $M$ together with an interpretation function giving a subset $\semantics{p}\subseteq M$ for each noun $p$.
We then define:
\begin{equation}
\label{semantics-all}
\Model\models \all{p}{q} \quadiff 
\semantics{p} \subseteq \semantics{q}.\\
\end{equation}
We employ standard model-theoretic notation and terminology. We say that $\Model$ \emph{satisfies} a sentence $\phi$, or that $\Model$ is a \emph{model of $\phi$}, when $\Model\models \phi$. A \emph{theory} is a set of sentences. 
Given a theory $\Gamma$, we write $\Model\models \Gamma$ to mean that $\Model$ satisfies every sentence in $\Gamma$;
naturally we say that \emph{$\Model$ satisfies $\Gamma$}, or that \emph{$\Model$ is a model of $\Gamma$}.
We then have the usual notion of logical consequence: given a theory $\Gamma$ and a sentence $\phi$,
we write $\Gamma\models\phi$ if every model of $\Gamma$ satisfies $\phi$.

We  match the semantics with a proof system.  Our system has two rules of inference, shown below:
\[
\infer[\axiom]{\all{x}{x}}{}
\qquad
\infer[\barbara]{\all{x}{z}}{\all{x}{y} & \all{y}{z}}
\]

In these rules, the material above the line is the set of \emph{premises}, and the sentence below
is the \emph{conclusion}. So $(\axiom)$ has no premises, and $(\barbara)$ has two. A \emph{substitution instance} of a rule is obtained by substituting nouns in $\bP$ for the variables $x$, $y$, and $z$.
We can then define the provability relation $\Gamma\proves \phi$: this means that there is a  tree whose root is $\phi$, and every  node in the tree is either a leaf  and belongs to $\Gamma$, or else it is the conclusion, and its children are the premises, of a substitution instance of one of our rules. The soundness/completeness theorem for this system states that $\Gamma\models\phi$ iff $\Gamma\proves \phi$.
For the proof, see~\cite{logic:moss08}.

\subsection{Syntax and semantics of the logics in this paper}  We are concerned not with $\langall$ but rather with a family of extensions of it.
We start with a set $\bP$ of nouns (just as above), and
also a set $\bR$, the \emph{verbs}.
Then we define \emph{terms} and \emph{sentences}
via the syntax below, where $p$ is any noun in $\bP$, $r$ is any verb in $\bR$, and $x$ and $y$ are any terms:
\begin{equation}\label{master-grammar}
\begin{split}
\mbox{terms} &\qquad p \mid \allterm{r}{x} \mid \someterm{r}{x} \mid \notterm{x} \\ 
\mbox{sentences} &\qquad  \all{x}{y} \mid \some{x}{y} 
\end{split}
\end{equation}

Note that we have recursion here, so terms can be nested, e.g.\ $(\allterm{r}{(\notterm{\someterm{s}{p}})})$.

For the semantics, we start with a model in our previous sense (to interpret the nouns),
and add interpretations of the verbs as binary relations: 
$\semantics{r} \subseteq M\times M$ for all $r\in \bR$.
Then we interpret our terms in a given model by recursion as follows:
\begin{align*}
\semantics{\allterm{r}{x}} &=  \set{m\in M :  \text{for all $n\in \semantics{x}$, $m \semantics{r} n$} } \\
\semantics{\someterm{r}{x}} &  = \set{m\in M :  \text{there is $n\in \semantics{x}$ such that $m \semantics{r} n$} } \\
\semantics{\notterm{x}} & = M\setminus\semantics{x}
\end{align*}
Thus, the interpretation of every term is a subset of $M$.
We define the truth-relation for sentences and models by generalizing (\ref{semantics-all}) above:
\begin{align*}
\Model\models \all{x}{y} & \quadiff  
\semantics{x} \subseteq \semantics{y} \\
\Model\models \some{x}{y} & \quadiff  
\semantics{x} \cap \semantics{y}\neq\emptyset 
\end{align*}

The basic languages in this paper are all sub-languages of the language just presented, which we call $\langfivesecond$.  And they are all extensions of the language we call $\langone$, which only has the sentence former (\all{x}{y}) and the term former (\allterm{r}{x}). There are three features in $\langfivesecond$ which are absent from $\langone$: the sentence former $(\some{x}{y})$, the term former (\someterm{r}{x}), and term complementation. We thus explore $2^3 = 8$ logical systems, obtained by all possible combinations of these features.  Those languages  are listed in 
the chart in Figure~\ref{language-chart}.
 We organize matters by studying 
$\lang_n$ and $\lang_{n.5}$  in Section $n$. Note that $\langone$ and $\langtwo$ are related in the same was as $\lang_n$ and $\lang_{n.5}$ for $n>2$, namely by adding the sentence former $(\some{x}{y})$.

We are interested in complete proof systems and the complexity of the consequence relation for each of these languages. By this we mean the computational complexity of the following decision problem: given a finite theory $\Gamma$ and a sentence $\phi$, output ``yes'' if $\Gamma\models \phi$ and ``no'' otherwise. Our results, which are summarized in the last two columns of Figure~\ref{language-chart}, will be explained in more detail in Section~\ref{section-results} below.

\begin{figure}[t]
\begin{mathframe}
\[
\begin{array}{|c||c|c||c|c|c|}
\hline
\mbox{language} & \mbox{term} & \mbox{sentence} &   \mbox{additions to} & \mbox{complexity of the} \\
 & \mbox{former(s)} & \mbox{former(s)}  & \mbox{syllogistic logic}&\mbox{consequence relation} \\
\hline
\hline
\langone & \allterm{r}{x} & \all{x}{y} &    \mbox{none} &  \mbox{$\PTIME$} \\
\hline
\langtwo & \allterm{r}{x} & \all{x}{y}, \some{x}{y} & \mbox{(\casesrule), or (\chains),}  &  \mbox{$\PTIME$} \\
  & & &   \mbox{or extra syntax} & \\
  \hline
\langthree & \allterm{r}{x}, \someterm{r}{x} & \all{x}{y}  &   \mbox{(\casesrulethree) \&\ (\casesruletwo)} 
   & \mbox{$\CONP$ complete}  \\
\hline
\langthreesecond & \allterm{r}{x}, \someterm{r}{x} & \all{x}{y},  \some{x}{y} &    \mbox{ (\casesrule) \&\ (\casesruleone)}
      & \mbox{$\CONP$ complete}  \\
\hline
\langfour & \allterm{r}{x}, \notterm{x} & \all{x}{y} & \mbox{extra syntax } \& & \mbox{$\CONP$ hard}   \\
& & & \mbox{schematic rules} & \\
\hline
\langfoursecond & \allterm{r}{x}, \notterm{x} & \all{x}{y}, \some{x}{y} & \mbox{extra syntax } \&\ (\raa) & \mbox{$\EXPTIME$} \\
& & & \&\ \mbox{schematic rules} & \\
\hline
\langfive & \allterm{r}{x}, \someterm{r}{x},\notterm{x} & \all{x}{y} & \mbox{} & \mbox{$\EXPTIME$ complete}   \\
\hline
\langfivesecond & \allterm{r}{x}, \someterm{r}{x},\notterm{x} & \all{x}{y}, \some{x}{y} & \mbox{individual variables} & \mbox{$\EXPTIME$ complete} \\
\hline
\end{array}
\]
\caption{Languages in this paper, given by section.\label{language-chart}}
\end{mathframe}
\end{figure}

\subsection{Related work} 

This paper is similar in spirit to~\cite{phmoss}, which also considered completeness and complexity results for decidable languages extending the basic syllogistic logic $\langall$. In fact, the largest language of this paper, $\langfivesecond$, is essentially equivalent to the largest language considered in~\cite{phmoss}, called $\mathcal{R}^{*\dag}$ there (see Propositions~\ref{fivesecond} and~\ref{fivesecondsecond} below). But there are several differences between the other languages in~\cite{phmoss} and in this paper.   First,~\cite{phmoss} 
allowed for complemented \emph{verbs}, contrary to what we do here.   Second,~\cite{phmoss} also explored
logical systems where the two terms in sentences like $(\all{x}{y})$ were not treated the same.  
For example, in two systems there, the subject noun phrase $x$ was required to be (negated) atomic.  Third, we allow nested terms, which are not part of the syntax in~\cite{phmoss}. The upshot is that
there is no overlap in the technical results from~\cite{phmoss} and this paper, except for the result on $\langfivesecond$ which we quote
in Section~\ref{section-five}.

In~\cite{logic:mcA+G92}, McAllester and Givan considered a language which is a slight extension of our language $\langthreesecond$. In  Section~\ref{section-three}, we essentially provide a new proof of their result that the consequence relation for this language is $\CONP$ complete. Our version of this result is slightly sharper, since we prove that the weaker language $\langthree$ is already $\CONP$ hard, and we also provide complete (though non-syllogistic) proof systems for $\langthree$ and $\langthreesecond$. 

We began this paper with $\langall$, which was introduced in~\cite{logic:moss08}. But the smallest language in Figure~\ref{language-chart} is $\langone$, which is the extension of $\langall$ by the term former $(\allterm{r}{x})$.
As a proof system for $\langone$, we take the rules $(\axiom)$ and $(\barbara)$ above (but we allow terms, not just nouns, to be substituted for the variables $x$, $y$, and $z$), together with a new rule: 
\[
\infer[\anti]{\all{(\allterm{r}{y})}{(\allterm{r}{x})}}{\all{x}{y}}
\]

It is easy to see that this proof system is sound (if $\Gamma\proves\phi$, then $\Gamma\models \phi$). Completeness was shown in~\cite{MossKruckman16}, which also contains completeness and complexity results for a number of related languages. The completeness proof originated in~\cite{Moss:LFL}, where it is also shown that the consequence relation for $\langone$ is in $\PTIME$. We reprove these results in Section~\ref{section-strongly} below, but using a more general framework described in Section~\ref{section-introduction-rules}. This framework unifies the $\PTIME$ results for $\langone$ and $\langtwo$ in Section~\ref{section-fourplace}, and allows us to obtain more precise negative results for the other languages in the paper, as we shall see.

 \subsection{Syllogistic proof systems and bounded completeness}
 \label{section-introduction-rules}
 
At this point, we wish to formally state what we mean by 
a syllogistic proof system. To state rules, we employ a language with \emph{noun variables} $p,q,\dots$, \emph{verb variables} $r, s,\dots$, \emph{term variables} $x,y,\dots$, and \emph{sentence variables} $\phi,\psi,\dots$. (In practice, none of our rules will make use of noun variables, and very few will make use of sentence variables.)

A \emph{term template} is defined as in (\ref{master-grammar}), but using noun and verb variables in place of nouns and verbs, and with an additional base case: a term variable is a term template. A \emph{sentence template} is defined as in (\ref{master-grammar}), but using term templates in place of terms, and with an addition option: a sentence variable is a sentence template. 

A \emph{syllogistic rule} $\rho$ consists of finitely many (possibly none) sentence templates as premises, and a single sentence template as a conclusion.
 A \emph{syllogistic proof system} is a finite set of syllogistic rules.  

We use $\proves$ to denote a syllogistic proof system.
Given a syllogistic proof system, we also use the symbol $\proves$ for the (standard)
provability relation, defined shortly. A \emph{substitution instance} of a rule $\rho$ is obtained by substituting nouns, verbs, terms, and sentences for all noun variables, verb variables, term variables, and sentence variables, respectively. A \emph{proof tree over a theory $\Gamma$} is a tree labeled by sentences, such that each node is either a leaf and belongs to $\Gamma$, or else it is the conclusion, and its children are the premises, of a substitution instance of one of the rules. We write $\Gamma\proves \phi$ if there is a proof tree over $\Gamma$ whose root is $\phi$.

We should mention that a syllogistic proof system is subject to some important limitations.
First, the system cannot include rules which allow for the withdrawal of assumptions, as in  \emph{reductio ad absurdum}, or proof by cases. (On the other hand, 
\emph{ex falso quodlibet} is a syllogistic rule; see~\cite{logic:moss08}.)  Second, the premises of each rule must be a fixed finite set of sentence templates;
the set of premises cannot be listed as a  \emph{schema}. 

Later in this paper, we shall see proof systems that are not syllogistic in our sense: To obtain completeness theorems for various logics, we need to add the rule (\casesrule) and its variants (\casesruleone), (\casesrulethree), and (\casesruletwo) in Section~\ref{sec-cases} and Section~\ref{section-three}, the rule (\raa) in Section~\ref{section-term-negation}, and the schema (\chains) in Section~\ref{section-chains}.

Next, we introduce a strengthening of the notion of syllogistic proof system, which ensures that the consequence relation $\Gamma\models \varphi$ is efficiently decidable, for any finite theory $\Gamma$ and any sentence $\varphi$.

\begin{definition}\label{def-confined}
Let $\lang$ be a language, equipped with a syllogistic proof system $\proves$, and let  $A$ be any set of sentences in $\lang$.
We write $\Gamma\provesA \phi$ if $\Gamma\proves  \phi$ via a proof tree $\TT$ with the property that all sentences in $\TT$ belong to 
 $A$.  
\end{definition}

\begin{example}\label{ex-pertinent}
This example is based on the observation that
\begin{equation}
\label{pertinent-prelim}
 \set{\all{x}{y} , \all{y}{z}} \proves \all{(\allterm{r}{z})}{(\allterm{r}{x}}).
\end{equation}
For example, here is a proof tree:
\[\infer[\barbara]{\all{(\allterm{r}{z})}{(\allterm{r}{x}})}
{
\infer[\anti]{\all{(\allterm{r}{z})}{(\allterm{r}{y})}}{\all{y}{z}}
&
\infer[\anti]{\all{(\allterm{r}{y})}{(\allterm{r}{x})}}{\all{x}{y}}
}
\]
Let $A$ be the set of all sentences $\psi$ such that every subterm of $\psi$ is in the set $\set{x,y,z,\allterm{r}{x},\allterm{r}{z}}$ (this is the set of subterms of the sentences appearing in (\ref{pertinent-prelim})).

Our tree above does not show that 
\begin{equation}
\label{moref}
 \set{\all{x}{y} , \all{y}{z}} \provesA \all{(\allterm{r}{z})}{(\allterm{r}{x}}).
 \end{equation}
The problem is that the term $(\allterm{r}{y})$ is not in $A$.
But the tree below does show (\ref{moref}):
\[
\infer[\anti]{\all{(\allterm{r}{z})}{(\allterm{r}{x})}}{
\infer[\barbara]
{\all{x}{z}}
{
\all{x}{y} & \all{y}{z}
}
}
\]
\end{example}

\begin{definition} A \emph{boundedly complete syllogistic proof system} for $\lang$ is a syllogistic proof system 
$\proves$ for $\lang$ which is sound, and such that that there exists a $\PTIME$-computable $f:\powfin(\lang)\to \powfin(\lang)$ such that whenever $\Gamma\models \phi$, then   $\Gamma\provesfGammaphi \phi$.\label{def-strongly}
\end{definition}

\begin{example}\label{ex-pertinent-two}
Suppose we are working with the language $\langone$. For any finite set of sentences $\Delta$, let $f(\Delta)$ be the set of all sentences $\varphi$ such that every subterm of $\varphi$ is a subterm of a sentence in $\Delta$. Note that $f$ is computable in polynomial time by enumerating the subterms of sentences in $\Delta$, and then forming all sentences $(\all{u}{v})$ such that $u$ and $v$ are on the list. 
In connection with Example~\ref{ex-pertinent}, the set $A$ there is exactly
$f(\set{ \all{x}{y} , \all{y}{z}, \all{(\allterm{r}{z})}{(\allterm{r}{x}})})$. In Proposition~\ref{prop-g} below, we are going to show that the proof system with rules ($\axiom$), ($\barbara$), and ($\anti$) is boundedly complete for $\langone$. The function $f$ could serve as a witness. But in order to simplify the proof, we use a slightly different function $g$.
\end{example}
  
\begin{theorem}
Fix a language $\lang$. Let $\proves$ be a boundedly complete syllogistic proof system for a language $\lang$.
Then $\proves$ is complete, and for any finite theory $\Gamma$ and any sentence $\varphi$, the
problem of deciding whether $\Gamma\models\phi$ is in $\PTIME$.
\label{theorem-ptime}  
 \end{theorem} 
  
\begin{proof} 
 This is essentially Appendix A of McAllester~\cite{McAllester93}.
Here is a sketch, based on this, and also on a parallel result in~\cite{phmoss} which was proved in the easier setting of syllogistic logics without complex terms.

Let $f:\powfin(\lang)\to \powfin(\lang)$ be a $\PTIME$-computable function such that $\Gamma\provesfGammaphi \phi$ whenever $\Gamma\models \phi$. 
In particular,  $\Gamma\models \phi$ implies that $\Gamma\proves \phi$.  So the proof system is complete. 

 For the $\PTIME$ decidability,  first compute $f(\Gamma\cup\set{\phi})$.   Call this set $A$. 
Let $X_0 = \Gamma\cap A$. We compute an increasing sequence of subsets of $A$ by induction. Given $X_n$, take each of the finitely many rules $\rho$ of the logic, and do the following:
 compute the set of all substitution instances of $\rho$ whose premises are all in $X_n$;
for each such substitution instance, if the conclusion $\psi$ belongs to $A$, then add $\psi$ to $X_{n+1}$.
 Continue until 
 the first $n^*$ such that $X_{n^*+1} =X_{n^*}$. Since all the $X_n$ are subsets of $A$, we have $n^* \leq 1 + |A|$. And $\Gamma\provesA \phi$ iff $\phi\in X_{n^*}$.
We take it as standard that all of this can be done in $\PTIME$.
 \end{proof}

\subsection{Example: Completeness and $\PTIME$ decidability for $\langone$}
\label{section-strongly}

We illustrate the application of
 Theorem~\ref{theorem-ptime} to $\langone$ with the syllogistic proof system consisting of the rules (\axiom), (\barbara), and (\anti). 
 
For any set of sentences $\Delta$,
let $T(\Delta)$ be the set of subterms of sentences in $\Delta$.
Let $T^+(\Delta)$ be $T(\Delta)$ together with the terms $(\allterm{r}{w})$
where $w\in T(\Delta)$ and where $r$ is a verb which appears in $\Delta$.

Let $g(\Delta)$ be the set of all sentences $(\all{u}{v})$,
where $u\in T(\Delta)$ and $v\in T^+(\Delta)$.
Note that when $\Delta$ is finite, $g$ is computable in $\PTIME$.

\begin{proposition}[\cite{Moss:LFL}]\label{prop-g}
If  $\Gamma\models \phi$, then $\Gamma\proves_{g(\Gamma\cup \set{\varphi})} \phi$.
Hence the consequence relation for $\langone$ is in $\PTIME$.
\end{proposition} 
 
\begin{proof}
Fix a finite theory $\Gamma$ and a sentence $\varphi$. We are going to save on some notation in this proof by writing $T$ for $T(\Gamma\cup\set{\phi})$,
$T^+$ for $T^+(\Gamma\cup\set{\phi})$,
and $A$ for $g(\Gamma\cup\set{\phi})$.

We make a model $\Model$ as follows.
The domain $M$ of the model is $T$.
The structure of the model is given by
\begin{align*}
t\in \semantics{p} &\quadiff  \Gamma\provesA\all{t}{p}\\
t\semantics{r} u &\quadiff \Gamma\provesA\all{t}{(\allterm{r}{u})}\\
\end{align*}
(Recall that the set $A$ is fixed as $g(\Gamma\cup \set{\varphi})$. When we write $\Gamma\provesA\psi$ in this proof, we are not changing the meaning of $A$ to $g(\Gamma\cup \set{\psi})$.)
 
\begin{claim}[Truth Lemma]
For all $a\in T$, 
\begin{equation}\label{indg}
t\in \semantics{a} \quadiff \Gamma\provesA\all{t}{a}. 
\end{equation} 
\end{claim}
\begin{proof}
The proof is by induction on $a$.   When $a$ is a noun, the assertion in (\ref{indg})
 is part of the definition of the model.
Assume (\ref{indg}) for $a$, and consider $(\allterm{r}{a})$.  Since $(\allterm{r}{a})\in T$, then $a\in T = M$ as well.   Note also that the sentence $(\all{a}{a})$ belongs to $A$, and so $\Gamma\provesA\all{a}{a}$ by $(\axiom)$. By induction, $a\in \semantics{a}$.

Let $t\in \semantics{\allterm{r}{a}}$. Since  $a\in \semantics{a}$,
we have $t\semantics{r} a$, and hence $\Gamma\provesA\all{t}{(\allterm{r}{a})}$ by definition of $\semantics{r}$. 

Conversely, assume that  
$\Gamma\provesA\all{t}{(\allterm{r}{a})}$.  Then $t\in T = M$. We show that $t\in \semantics{\allterm{r}{a}}$.
For this, let $b\in\semantics{a}$, so by induction, $\Gamma\provesA\all{b}{a}$.
 We must show that $t\semantics{r} b$, equivalently $\Gamma\provesA\all{t}{(\allterm{r}{b})}$.
Note that $b\in M=T$, so $(\allterm{r}{b})\in T^+$.
The key point is that the sentence $\all{(\allterm{r}{a})}{(\allterm{r}{b})}$ belongs to $A$, since $(\allterm{r}{a})\in T$ and $(\allterm{r}{b})\in T^+$.  
Using (\anti) and $\Gamma\provesA\all{b}{a}$, we see that 
 $\Gamma\provesA\all{(\allterm{r}{a})}{(\allterm{r}{b})}$.   Then by (\barbara), 
 $\Gamma\provesA\all{t}{(\allterm{r}{b})}$.
\end{proof}

We continue by showing that $\Model\models\Gamma$.
For this, take any sentence $(\all{u}{v})$ in $\Gamma$. 
Let $t\in\semantics{u}$.  By (\ref{indg}), $\Gamma\provesA\all{t}{u}$. Now $t$, $u$, and $v$ are in $T$, so the sentences $(\all{u}{v})$ and $(\all{t}{v})$ are in $A$. By (\barbara),  $\Gamma\provesA\all{t}{v}$. 
By (\ref{indg}) again, $t\in\semantics{v}$.   Since $t$ was arbitrary, we have shown that $\semantics{u}\subseteq \semantics{v}$, and $\Model\models \all{u}{v}$. 

Since $\Model\models\Gamma$ and $\Gamma\models \phi$, the sentence $\phi$ holds in $\Model$.
Let us write $\phi$ as $(\all{a}{b})$.
Then $a\in T = M$ and the sentence $(\all{a}{a})$ belongs to $A$, and so $\Gamma\provesA\all{a}{a}$ by (\axiom).
Thus, $a\in\semantics{a}$ by (\ref{indg}).   So  $a\in\semantics{b}$.   By (\ref{indg}) again, $\Gamma\provesA\all{a}{b}$. 
This concludes the proof.
\end{proof}

\subsection{Results}
\label{section-results}

Our two main themes are trade-offs between expressive power
and computational complexity, and also the variety of devices that one can add on top of pure syllogistic logic in order to obtain 
sound and complete proof systems. Our results are summarized in Figure~\ref{language-chart}. 
    
We begin with a negative result:  
$\langtwo$ has  no sound and complete syllogistic proof system.
The argument for this is combinatorial, and in outline it is based on a similar result in~\cite{phmoss} for $\cR$.
Nevertheless, there are  proof systems which capture the consequence relation of $\langtwo$.
Most of Section~\ref{sec-langtwo} is devoted to several such logics, each of which extends syllogistic logic in a different way. 
One way is to add a rule (\casesrule) which enables proof by cases, 
 another uses a schema of rules called (\chains) (thus there are infinitely many rules, but
the set of rules is an easily-defined set), and the last is to extend the syntax in such a way that the extension, called $\langtwo^+$, does
have a sound and complete syllogistic proof system. In fact, this system is boundedly complete, and by Theorem~\ref{theorem-ptime}, the consequence relation of $\langtwo^+$, and hence of its sublanguage $\langtwo$, is in $\PTIME$.

Notice that what we will show is that $\langtwo$ has no sound and complete syllogistic proof system, but the larger language $\langtwo^+$ does have such a proof system.
The first example of this phenomenon is in~\cite{prattHartmann2014}.

The lower bounds on complexity established in the rest of the paper show that
(assuming {\sc P $\neq$ NP}) none of the other languages admit boundedly complete syllogistic proof systems. 

The languages $\langthree$ and $\langthreesecond$ extend $\langone$ and $\langtwo$ with the term former $(\someterm{r}{x})$. We show that the smaller language $\langthree$ has a consequence relation which is $\CONP$ hard, via a reduction from the one-in-three positive  $3$-SAT problem. On the other hand, a finite countermodel construction shows that the consequence relation of the larger language $\langthreesecond$ is in $\CONP$. It follows that the consequence relation for both languages are $\CONP$ complete. We give sound and complete proof systems for these logics which are not syllogistic: they use variants of the rule (\casesrule).

The languages $\langfour$ and $\langfoursecond$ extend $\langone$ and $\langtwo$ with term complementation.   As with $\langthree$, we show that the 
consequence relation for $\langfour$  is $\CONP$ hard, via a reduction from $3$-SAT. But this time we leave open the question of $\CONP$ completeness; the upper bound of $\EXPTIME$ comes from the known complexity of the larger language $\langfivesecond$. We also leave open the problem of formulating proof systems for $\langfour$ and $\langfoursecond$ in their original syntax. 
Instead, we present a completeness result for an extension $\langfoursecond^+$ of $\langfoursecond$. The proof system 
and the completeness result are comparatively simple, though the proof system is decidedly non-syllogistic: it includes a form of \emph{reductio ad absurdum} (\raa), as well as several rule schemas, and the syntax is not even finitary, since there is an infinite family of sentence formers. This proof system restricts, by dropping $(\raa)$, to a sound and complete proof system for the corresponding extension $\langfour^+$ of $\langfour$.

The largest languages in this paper are 
 $\langfive$ and $\langfivesecond$.  We have the least to say about them, mostly because 
 $\langfivesecond$ has been studied (in a different but equivalent formulation, called $\mathcal{R}^{*\dag}$) in~\cite{phmoss},
 and the complexity result for it from~\cite{phmoss} also holds for $\langfive$, as we shall see. The sound and complete proof system for $\mathcal{R}^{*\dag}$ from~\cite{phmoss} (which is non-syllogistic due to its use of individual variables) can be adapted to a similar proof system for $\langfivesecond$, but we do not make that explicit here. We leave open the problem of formulating a proof system for $\langfive$. 
 
One might guess that when we explore a partial order of logics, stronger logics are are harder to work with and to prove completeness for.  But this is not always the case.  
One reason:  as the logics get stronger, they include more and more features of first-order logic and are thus easier to analyze, due to our experience with first-order logic.  A second reason: sometimes adding to the syntax of a logic restores harmony in some way, thereby making it easier for us to work with.      The examples of $\langtwo^+$, $\langfour^+$, and $\langfoursecond^+$ emphasize the fact that there is not a monotone relationship between the strength of logical systems and their elegance, or between their strength and the difficulty of proving completeness. 
It is an open problem to develop a general theory which could explain this phenomenon. For example, why it is that some logical systems have (boundedly complete) syllogistic proof systems ($\langtwo^+$), while others do not ($\langtwo$)?   What we have at present are some ad hoc results, and much remains to be done.

\section{$\langtwo$: Adding the sentence former $(\some{x}{y})$ to $\langone$}\label{sec-langtwo}

The main results on $\langtwo$ are (1) it has no finite sound and complete syllogistic proof system; (2) nevertheless there are several non-syllogistic devices which allow us to obtain sound and complete proof systems;
(3) alternatively, we can extend the syntax of $\langtwo$ to a larger language $\langtwo^+$ in such a way that $\langtwo^+$ has
a boundedly complete syllogistic proof system;  and
(4) as a result of this last point, the consequence relation $\Gamma\models\phi$
for $\langtwo^+$ (and hence $\langtwo$)
 is
in $\PTIME$.

\subsection{The base system $\proves_0$}\label{sec-basesystem}

We begin with a proof system in Figure~\ref{baserules} that will be used in this section and beyond.

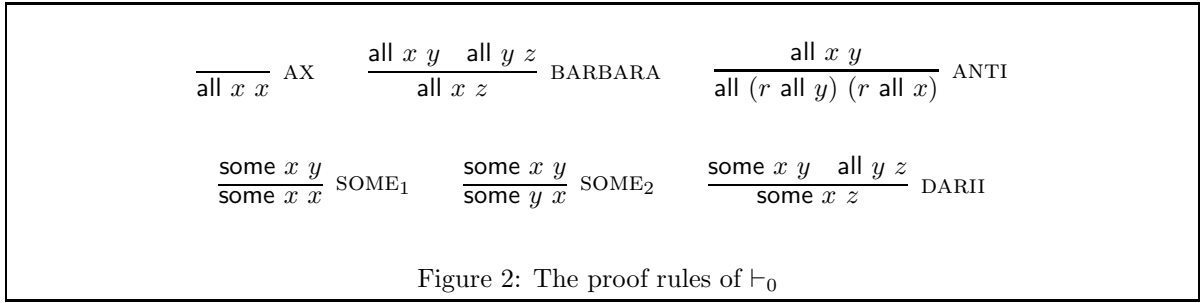
\begin{figure}[h]
\begin{mathframe}
\begin{gather*}
\infer[\axiom]{\all{x}{x}}{} \qquad
\infer[\barbara]{\all{x}{z}}{\all{x}{y}  & \all{y}{z}}\qquad
\infer[\anti]{\all{(\allterm{r}{y})}{(\allterm{r}{x})}}{\all{x}{y}}  
\\  \\
\infer[\someone]{\some{x}{x}}{\some{x}{y}}\qquad
\infer[\sometwo]{\some{y}{x}}{\some{x}{y}}\qquad
\infer[\darii]{\some{x}{z}}{\some{x}{y} & \all{y}{z}}
\end{gather*}
\caption{The proof rules of $\proves_0$ \label{baserules}}
\end{mathframe}
\end{figure}

It is easy to check that the rules in Figure~\ref{baserules} are sound. Theorem~\ref{theorem-allcomplete} below shows that they are also complete for conclusions of the form $(\all{x}{y})$. Since the rules $(\axiom)$, $(\barbara)$, and $ (\anti)$ are complete for $\langone$, this result can be interpreted as saying that $\langtwo$ is a conservative extension of $\langone$. It is also possible to give a direct model-theoretic proof of this conservativity result.

In many places in this paper, we will work with a set $\Terms$ of terms. We will always assume that
such a set $\Terms$ is closed under subterms.   We could usually take $\Terms$ to be the set of all terms in the language under study.
But when we use $\Terms$ to build a model (as we do just below), that model will be 
 infinite when $\Terms$ is infinite. Working with a finite set $\Terms$  allows us to build finite models in many situations.

We write $\Gamma\proves_0 \varphi$ if there is a proof of the sentence $\varphi$ from the theory $\Gamma$ using the rules in Figure~\ref{baserules}. 
We also write
\[
x \leq y
\]
to mean that $\Gamma\proves_0\all{x}{y}$.  We use this notation because $\leq$ is a preorder, due to 
(\axiom) and (\barbara).  Please note that $\Gamma$ is left off of this notation.

\paragraph{The first canonical model of a theory $\Gamma$}
Let $\Gamma$ be a theory, let $\Terms$ be a set of terms (closed under subterms as usual), and let $M$ be the set of unordered pairs $\set{t,u}$ of terms from $\Terms$. This includes singletons $\set{t} = \set{t,t}$. We define a model $\Model(\Gamma,\Terms)$ with domain $M$ by setting
\begin{equation}
\label{canonical-model}
\begin{split}
\set{t,u}\in \semantics{p} &\quadiff t \leq p \mbox{ or } u \leq p \\
\set{t,u}\semantics{r}\set{v,w}  & \quadiff \mbox{for some $a\in \set{t,u}$ and $b\in\set{v,w}$, $a\leq \allterm{r}{b}$}
\end{split}
\end{equation}

\begin{lemma}[Truth Lemma]\label{first-Truth-Lemma}
In $\Model(\Gamma,\Terms)$, for all terms $x\in \Terms$,
\[
\semantics{x} =  \set{\set{t,u}\in M :  t \leq x \mbox{ or } u \leq x }.
\]
\end{lemma}
 \begin{proof}
 By induction on $x$.   For a noun $p\in \bP$, this is by definition of the model.
For a term of the form $(\allterm{r}{x})\in \Terms$, note that $x\in \Terms$, since $\Terms$ is closed under subterms. 

Suppose that 
$\set{t,u}\in \semantics{\allterm{r}{x}}$. Since $x\leq x$ by $(\axiom)$, the induction hypothesis implies $\set{x}\in\semantics{x}$.
So $\set{t,u}\semantics{r}\set{x}$.  
By the definition of $\semantics{r}$, either
$t\leq \allterm{r}{x}$, or else $u\leq \allterm{r}{x}$.

Conversely, fix $\set{t,u}\in M$, and suppose that  (without loss of generality)
$t\leq \allterm{r}{x}$.
 Let $\set{v,w}\in M$ be an element of $\semantics{x}$.  By the induction hypothesis, we have (without loss of generality) $v\leq x$. By (\anti), 
 $\allterm{r}{x} \leq \allterm{r}{v}$.  By $(\barbara)$, $t \leq \allterm{r}{v}$. 
So $\set{t,u}\semantics{r}\set{v,w}$, and hence $\set{t,u}\in \semantics{\allterm{r}{x}}$.
 \end{proof}
 
 \begin{lemma}
If $(\all{x}{y})\in \Gamma$, then $\Model(\Gamma,\Terms)\models \all{x}{y}$. If $x$ and $y$ are  any terms in $\Terms$, we have $\Model(\Gamma,\Terms)\models \some{x}{y}$. As a consequence, if $\Terms$ contains all subterms of sentences in $\Gamma$, then $\Model(\Gamma,\Terms)\models \Gamma$.
\label{first-canonical-model}
\end{lemma}
\begin{proof}
Suppose $(\all{x}{y})\in \Gamma$. If $\set{t,u}\in \semantics{x}$ in $\Model(\Gamma,\Terms)$, then by Lemma~\ref{first-Truth-Lemma} (without loss of generality), $t\leq x$. But then $t\leq y$ by $(\barbara)$, so $\set{t,u}\in \semantics{y}$ by Lemma~\ref{first-Truth-Lemma}, and $\Model(\Gamma,\Terms)\models \all{x}{y}$. 

Now suppose $x,y\in \Terms$. Since $x\leq x$ and $y\leq y$, we have $\set{x,y}\in \semantics{x}\cap \semantics{y}$ by Lemma~\ref{first-Truth-Lemma}, so $\Model(\Gamma,\Terms)\models \some{x}{y}$.

If $\Terms$ contains all subterms of sentences in $\Gamma$, then $\Model(\Gamma,\Terms)\models \all{x}{y}$ whenever $(\all{x}{y})\in \Gamma$ by the first assertion, and $\Model(\Gamma,\Terms)\models \some{x}{y}$ whenever $(\some{x}{y})\in \Gamma$ by the second assertion.
\end{proof}

\begin{theorem}\label{theorem-allcomplete}
If $\Gamma\models \all{x}{y}$, then $\Gamma\proves_0 \all{x}{y}$. Moreover, the proof only uses the rules $(\axiom)$, $(\barbara)$, and $(\anti)$. 
\end{theorem}
\begin{proof}
Choose $\Terms$ so that it contains $x$, $y$, and all terms in $\Gamma$. If $\Gamma\models \all{x}{y}$, then $\Model(\Gamma,\Terms) \models \all{x}{y}$ by Lemma~\ref{first-canonical-model}. Then $\set{x}\in \semantics{x}\subseteq \semantics{y}$ by $(\axiom)$ and Lemma~\ref{first-Truth-Lemma}. By Lemma~\ref{first-Truth-Lemma} again, $\Gamma\proves_0 \all{x}{y}$.

To see that the proof uses only the rules $(\axiom)$, $(\barbara)$, and $(\anti)$, we just need to examine the rules in the proof system and note that no rule which produces a conclusion of the form $(\all{x}{y})$ has a premise of the form $(\some{a}{b})$. 
\end{proof}

The proof of Theorem~\ref{theorem-allcomplete} shows that if $\Gamma\not\proves_0 \all{x}{y}$, then there is a countermodel of size $O(n^2)$, where $n$ is the complexity of $\Gamma\cup \set{\all{x}{y}}$. However, as noted in Lemma~\ref{first-canonical-model}, $\Model(\Gamma,\Terms)$ satisfies every sentence $(\some{x}{y})$ with $x,y\in \Terms$. To obtain a countermodel for sentences of this form, we will look at a submodel of $\Model(\Gamma,\Terms)$.

\paragraph{The second canonical model of a theory $\Gamma$}
Let $M'$ be the set of unordered pairs $\set{t,u}$ of terms in $\mathbb{T}$ such that $\Gamma\proves_0\some{t}{u}$. 
Note that we allow $t = u$.   Define a model $\Model'(\Gamma,\Terms)$ with domain $M'$ just as in (\ref{canonical-model}). 

The proof system $\proves_0$ is not complete for sentences of the form $(\some{x}{y})$, but we can prove a partial completeness result under the following additional hypothesis on $\Gamma$, a form of which was first introduced by McAllester and Givan~\cite{logic:mcA+G92}. 

\begin{definition} 
We say that \emph{$\Gamma$ determines existentials for $\Terms$} if, for all verbs $r\in\bR$ and all terms $x,y\in \Terms$, either $\Gamma\proves_0 \some{x}{x}$ or $\Gamma\proves_0 \all{y}{(\allterm{r}{x})}$.
\end{definition}

The intuition behind this definition is that in any model $\Model$, for any term $x$, either $\Model\models\some{x}{x}$, or $\semantics{x} = \emptyset$, in which case $\semantics{\allterm{r}{x}} = M$ for any verb $r$, and $\Model \models \all{y}{(\allterm{r}{x})}$ for any term $y$.

\begin{lemma}\label{lemma-determines-existentials-Truth-Lemma}
Suppose $\Gamma$ determines existentials for $\Terms$. Then:
\begin{enumerate}[(1)]
\item In $\Model'(\Gamma,\Terms)$, for all terms $x\in \Terms$, $\semantics{x} =  \set{\set{t,u}\in M' : t\leq x \mbox{ or } u\leq x}$.
\item If $(\all{x}{y})\in \Gamma$, then $\Model'(\Gamma,\Terms)\models \all{x}{y}$.
\item If $x,y\in \Terms$ and $(\some{x}{y})\in \Gamma$, then $\Model'(\Gamma,\Terms)\models \some{x}{y}$. 
\end{enumerate}
As a consequence, if $\Gamma$ is any theory, $\Terms$ contains all subterms of sentences in $\Gamma$, and $\Gamma^*\supseteq \Gamma$ is a theory which determines existentials for $\Terms$, then $\Model'(\Gamma^*,\Terms)\models \Gamma$.
\end{lemma}
\begin{proof}
The proof of (1) is exactly like the proof of Lemma~\ref{first-Truth-Lemma}, with the following adjustment: If $\set{t,u}\in \semantics{\allterm{r}{x}}$ and $\set{x}\in M'$, the proof in Lemma~\ref{first-Truth-Lemma} goes through as written. But if $\set{x}\notin M'$, then $\Gamma\not\proves_0 \some{x}{x}$, so $t\leq \allterm{r}{x}$ (and also $u\leq \allterm{r}{x}$), since $\Gamma$ determines existentials for $\Terms$.

The proof of (2) is exactly as in the proof of Lemma~\ref{first-canonical-model}. 

The proof of (3) is also exactly as in the proof of Lemma~\ref{first-canonical-model}, with the following adjustment: We need to use the fact that $(\some{x}{y})\in \Gamma$ to see that $\set{x,y}\in M'$. 

Putting this together, suppose $\Gamma$ is any theory, $\Terms$ contains all subterms of sentences in $\Gamma$, and $\Gamma^*\supseteq\Gamma$ determines existentials for $\Terms$. If $(\all{x}{y})\in \Gamma\subseteq \Gamma^*$, then $\Model'(\Gamma^*,\Terms)\models \all{x}{y}$. And if $(\some{x}{y})\in \Gamma\subseteq \Gamma^*$, then since $x,y\in \Terms$, $\Model'(\Gamma^*,\Terms)\models \some{x}{y}$. 
\end{proof}

\begin{theorem} \label{theorem-detextcompleteness}
Suppose that $\Terms$ contains $x$, $y$, and all subterms of sentences in $\Gamma$, and $\Gamma^*\supseteq \Gamma$ is a theory which determines existentials for $\Terms$. If $\Gamma\models \some{x}{y}$, then $\Gamma^*\proves_0 \some{x}{y}$. 
\end{theorem}

\begin{proof}    
Since $\Gamma\subseteq \Gamma^*$, $\Model'(\Gamma^*,\Terms)\models \Gamma$ by Lemma~\ref{lemma-determines-existentials-Truth-Lemma}, so $\Model'(\Gamma^*,\Terms)\models \some{x}{y}$. Suppose $\set{t,u}\in \semantics{x}\cap \semantics{y}$. Since $\set{t,u}\in M'$, $\Gamma^*\proves_0 \some{t}{u}$. And by Lemma~\ref{lemma-determines-existentials-Truth-Lemma}, $\Gamma^*\proves_0 \all{v}{x}$ and $\Gamma^*\proves_0\all{w}{y}$ for some $v,w\in\set{t,u}$.
 By analyzing the four cases and using $(\someone)$, $(\sometwo)$, and $(\darii)$,
 we see that  $\Gamma^* \proves_0 \some{x}{y}$.
\end{proof}

\begin{corollary}
If $\Terms$ contains all subterms of sentences in $\Gamma\cup \set{\phi}$ and $\Gamma$ determines existentials for $\Terms$, then $\Gamma\models \varphi$ if and only if $\Gamma\proves_0 \varphi$.
\end{corollary}
\begin{proof}
The implication $\Gamma\proves_0\varphi$ implies $\Gamma\models \varphi$ is just soundness of the rules in $\proves_0$. The converse follows immediately from Theorems~\ref{theorem-allcomplete} and~\ref{theorem-detextcompleteness}, taking $\Gamma^* = \Gamma$ in Theorem~\ref{theorem-detextcompleteness}.
\end{proof}

We conclude this section with two proof-theoretic observations about the system $\proves_0$. 

If $\rvec = r_1,\dots,r_n$ is a sequence of verbs and $x$ is a term, we use the notation $(\allterm{\rvec}{x})$  for the term $(\allterm{r_1}{(\allterm{r_2}{(\dots (\allterm{r_n}{x}))})})$. When $\rvec$ is the empty sequence, $(\allterm{\rvec}{x}) = x$.

If $\psi = \all{u}{v}$, we define 
\[
\Anti(\rvec,\psi) = \begin{cases} 
\all{(\allterm{\rvec}{u})}{(\allterm{\rvec}{v})} & \text{if the length of $\rvec$ is even}\\
\all{(\allterm{\rvec}{v})}{(\allterm{\rvec}{u})} & \text{if the length of $\rvec$ is odd}.
\end{cases}
\]
Note that $\set{\psi}\proves_0 \Anti(\rvec,\psi)$ by repeated applications of $(\anti)$.

\begin{definition}\label{Gamma-sequence}
Let $\Gamma$ be a theory. A \emph{$\Gamma$-sequence} is a finite sequence of terms $t_1,\dots,t_n$, such that for all $1\leq i <n$ there is a sentence $\psi = (\all{a}{b})\in \Gamma$ and a sequence of verbs $\rvec$ such that $(\all{t_i}{t_{i+1}}) = \Anti(\rvec,\psi_i)$. 
\end{definition}

\begin{lemma}\label{lemma-allproof} $\Gamma\proves_0 \all{x}{y}$ if and only if there is a $\Gamma$-sequence of terms $t_1,\dots,t_n $
such that $x = t_1$ and $y = t_n$.
\end{lemma}
\begin{proof}
Suppose $x = t_1,\dots,t_n = y$ is a $\Gamma$-sequence. If $n = 1$, then $x = y$, and $\Gamma\proves_0 \all{x}{y}$ by $(\axiom)$. If $n>1$, then for all $1\leq i < n$, $\Gamma\proves_0 \all{t_i}{t_{i+1}}$ by repeated applications of $(\anti)$. And by repeated applications of $(\barbara)$, $\Gamma\proves_0 \all{x}{y}$.

We prove the converse by induction on the height of the proof tree. In the base case, $(\all{x}{y})\in \Gamma$, and $x, y$ is a $\Gamma$-sequence
as desired.

Case 1: If the root of the proof tree is \[\infer[\axiom]{\all{x}{x}}{}\] then $x$ is a $\Gamma$-sequence from $x$ to $x$.

Case 2: If the root of the proof tree is \[\infer[\barbara]{\all{x}{z}}{\all{x}{y} & \all{y}{z}}\] then by induction we have $\Gamma$-sequences $t_1,\dots, t_n$ and $t_1',\dots,t_m'$ with $x = t_1$, $y = t_n = t_1'$, and $z = t_m'$. Then $t_1,\dots,t_{n-1}, t_1',\dots,t_m'$ is a $\Gamma$-sequence from $x$ to $z$.

Case 3: If the root of the proof tree is \[\infer[\anti]{\all{(\allterm{r}{y})}{(\allterm{r}{x})}}{\all{x}{y}}\] then by induction we have a $\Gamma$-sequence $t_1,\dots, t_n$ with $x = t_1$ and $y = t_n$. Then $(\allterm{r}{t_n}),\dots,(\allterm{r}{t_1})$ is a $\Gamma$-sequence from $(\allterm{r}{y})$ to $(\allterm{r}{x})$.
\end{proof}

\begin{lemma}\label{lemma-someproof}
$\Gamma\proves_0 \some{x}{y}$ if and only if there is a sentence $(\some{t_1}{t_2})\in \Gamma$ such that  $\Gamma\proves_0 \all{t_i}{x}$ and $\Gamma\proves_0 \all{t_j}{y}$, for some $i,j\in \set{1,2}$. 
\end{lemma}
\begin{proof}
Suppose there is a sentence $\some{t_1}{t_2}\in \Gamma$ such that $\Gamma\proves_0 \all{t_i}{x}$ and $\Gamma\proves_0 \all{t_j}{y}$, for some $i,j\in \set{1,2}$. Then applying $(\someone)$ or $(\sometwo)$ if necessary, $\Gamma\proves_0 \some{t_i}{t_j}$ By two applications of $(\darii)$ and $(\sometwo)$, $\Gamma\proves_0 \some{x}{y}$. 

We prove the converse by induction on the height of the proof tree. In the base case, when $(\some{x}{y})\in\Gamma$,  we have $\Gamma\proves_0 \all{x}{x}$ and $\Gamma\proves_0 \all{y}{y}$ by $(\axiom)$.

Case 1: If the root of the proof tree is \[\infer[\someone]{\some{x}{x}}{\some{x}{y}}\]then by induction there is a sentence $(\some{t_1}{t_2})\in \Gamma$ such that $\Gamma\proves_0 \all{t_i}{x}$ for some $i\in \set{1,2}$.

Case 2: If the root of the proof tree is \[\infer[\sometwo]{\some{y}{x}}{\some{x}{y}}\]  then by induction there is a sentence $(\some{t_1}{t_2})\in \Gamma$ such that $\Gamma\proves_0 \all{t_i}{y}$ and $\Gamma\proves_0 \all{t_j}{x}$, for some $i,j\in \set{1,2}$. 

Case 3: If the root of the proof tree is \[\infer[\darii]{\some{x}{z}}{\some{x}{y} & \all{y}{z}}\]  then by induction there is a sentence $(\some{t_1}{t_2})\in \Gamma$ such that $\Gamma\proves_0 \all{t_i}{x}$ and $\Gamma\proves_0 \all{t_j}{y}$, for some $i,j\in \set{1,2}$. But also $\Gamma\proves_0 \all{y}{z}$, so $\Gamma\proves_0 \all{t_j}{z}$ by $(\barbara)$.
\end{proof}

\subsection{No sound and complete syllogistic proof system for $\langtwo$}
\label{section-no-finite}

In this section, we prove that there is no sound and complete syllogistic proof system
for $\langtwo$.
 This suggests that we need non-syllogistic devices like those which we shall see in coming sections of this paper.
But this talk of rules being ``needed'' is not precise, and at the end of the day,  it is not quite what we shall prove.
At the same time, what we do prove is in a real way stronger than the statement above.   So we need to make all of this precise.
 
We return to our discussion of syllogistic rules in Section~\ref{section-introduction-rules}.
Every syllogistic proof system  defines a \emph{provability relation}
between theories and sentences.   In this section, we write this relation 
as  $\Gamma\proves^*\phi$.  To be \emph{sound}, we require that if
$\Gamma\proves^*\phi$, then $\Gamma\models\phi$. To be \emph{complete}, we require that 
if $\Gamma\models\phi$, then also $\Gamma\proves^*\phi$.

The   \emph{degree $k$ consequence relation} $\models_k$
is the relation between sets $\Gamma$ and sentences $\phi$ defined as follows: $\Gamma\models_k \phi$
if there is a finite tree with nodes labeled by sentences, such that each node is either a leaf and in $\Gamma$, or else is a  sentence $\phi$ 
with children $\psi_1$, $\ldots$, $\psi_j$ for some $j\leq k$, and such that  $\set{\psi_1, \ldots, \psi_j}\models \phi$.
If we have a sound syllogistic proof system $\proves^*$, then since it has only finitely many rules, each with finitely many premises, there is a number $k$ (the maximum number of premises in any rule in $\proves^*$) such that
whenever $\Gamma\proves^*\phi$, we also have $\Gamma\models_k\phi$.

\begin{theorem}
For all $n$, there is a theory $\Gamma_{n+1}$ and a sentence $\phi$ such that $\Gamma_{n+1} \models\phi$, and $\Gamma_{n+1}\not\models_n\phi$. As a consequence, there is no sound and complete syllogistic proof system for $\langtwo$.
\label{theorem-stronger}
\end{theorem}

Note that the first statement does not refer to proof systems in any way.  It is completely semantic. But it immediately implies the negative result about proof systems.

\paragraph{The sets $\Gamma_n$}
For all $n$, let $\bP = \set{a,b}$, let $\bR = \set{r_1,\dots,r_n}$, and let $\Gamma_n$ be the following theory:
\begin{align*} 
\alpha &= \some{(\allterm{r_1}{(\allterm{r_1}{a})})}{(\allterm{r_1}{(\allterm{r_1}{a})})}\\
\phi_1 &=  \all{(\allterm{r_1}{b})}{(\allterm{r_2}{(\allterm{r_2}{a})})}\\
\phi_2 &=  \all{(\allterm{r_2}{b})}{(\allterm{r_3}{(\allterm{r_3}{a})})}\\
&\quad \vdots \\
\phi_i &=  \all{(\allterm{r_i}{b})}{(\allterm{r_{i+1}}{(\allterm{r_{i+1}}{a})})}\\
&\quad \vdots \\
\phi_{n-1} &=  \all{(\allterm{r_{n-1}}{b})}{(\allterm{r_n}{(\allterm{r_n}{a})})}\\
\omega &= \all{(\allterm{r_n}{b})}{a}
\end{align*}
We will use $\Gamma_n$ as a recurring example in the forthcoming sections, to demonstrate proof systems.

If $r$ is a verb, then an \emph{$r$-king} in a model $\Model$ is an element $x\in M$ such that for all $y\in M$, $x\semantics{r} y$. 

\begin{lemma}
$\Gamma_n\models\some{a}{a}$.
\label{lemma-some-a-a}
\end{lemma}

\begin{proof} 
Let $\Model\models\Gamma_n$.  If $\semantics{a} \neq \emptyset$, we are done.
So we shall assume that $\semantics{a} = \emptyset$.  Then for any verb $r$, $\semantics{\allterm{r}{a}} = M$.
By $\alpha$, $\Model$ contains an $r_1$-king $x$. Then $\phi_1$ implies that $x$ is also an $r_2$-king.  Continuing by induction, $\phi_i$ implies that $x$ is an $r_{i+1}$-king. Finally, $x$ is an $r_n$-king, and $\omega$ implies that $x\in \semantics{a}$, which is a contradiction.
\end{proof}

\begin{lemma} If $\Gamma_n\models \all{u}{v}$, then either $u = v$, or $(\all{u}{v}) = \Anti(\rvec,\psi)$, for some sequence $\rvec$ and some sentence $\psi\in \Gamma_n$. 
\label{lemma-Anti}
\end{lemma}

\begin{proof}
By Theorem~\ref{theorem-allcomplete}, $\Gamma_n\proves_0 \all{u}{v}$, and by Lemma~\ref{lemma-allproof}, there is a $\Gamma_n$-sequence $u = t_1,\dots,t_m = v$ of length $m$. 
If $m =1$, then $u = v$ and we are done.  
If $m = 2$, then $(\all{u}{v}) = \Anti(\rvec,\psi)$ for some sequence $\rvec$ and some $\psi\in \Gamma_n$;
we are again done.

It remains to prove a contradiction from $m \geq 3$.   Suppose that we have  $\psi, \theta\in \Gamma_n$ and 
sequences of verbs $\rvec$ and $\svec$ such that $(\all{t_1}{t_2}) = \Anti(\rvec,\psi)$ and $(\all{t_2}{t_3}) = \Anti(\svec,\theta)$.
Notice that $t_2$ must contain either $a$ or $b$.  Assume that $t_2$ contains $a$. The argument when $t_2$ contains $b$ is similar. 
Let $m$ be the   the number of verbs in $t_2$, the second term of $\Anti(\rvec,\psi)$.
Since it is $a$ rather than $b$ which occurs in  the second term of $\Anti(\rvec,\psi)$, $m$ is even.
We see this by examining the sentences in $\Gamma_n$.
And let $n$ be the   the number of verbs in $t_2$, the first term of $\Anti(\svec,\theta)$.  
This time, we see that $n$ is odd.
But $m = n$, since  the second term of $\Anti(\rvec,\psi)$ \emph{is} the first term of $\Anti(\svec,\theta)$. 
   This is a contradiction.
\end{proof}

\paragraph{The sets $\Delta_{n,i}$} For all $n$ and all $1 \leq i \leq n$,
let the theory $\Delta_{n,i}$ be given by
\[\Delta_{n,i} = \left\{ \begin{array}{ll} (\Gamma_n\setminus \set{\phi_i})
    & \mbox{if $1 \leq i < n$} \\
 (\Gamma_n\setminus \omega) & \mbox{if $ i = n$}
 \end{array}
\right.
\]

\begin{lemma} Suppose that for some $1\leq i\leq n$ and some terms $u$ and $v$, \[\Delta_{n,i}\models \some{u}{v}.\]  Then $u = v = (\allterm{r_1}{(\allterm{r_1}{a})})$, so that $(\some{u}{v}) = \alpha$. 
\label{lemma-r1-r1}
\end{lemma}

\begin{proof} It suffices to show that for every term $t\neq (\allterm{r_1}{(\allterm{r_1}{a})})$, there is a model of $\Delta_{n,i}$ in which $\semantics{t} = \emptyset$. Indeed, then $\Delta_{n,i}\not\models \some{u}{v}$ when $u = t$ or $v = t$.

We proceed by cases. In every case except $t = a$, we actually obtain a model of $\Gamma_n$ in which $\semantics{t} = \emptyset$. Since the models $\Model_4$ of $\Gamma_n$ constructed in Case 4 have $\Model_4\not\models \some{a}{(\allterm{r_1}{(\allterm{r_1}{a})})}$, this implies the additional result that if $\Gamma_n\models \some{u}{v}$, then $u = v = (\allterm{r_1}{(\allterm{r_1}{a})})$ or $u = v = a$.

\emph{Case 1:} $t = (\allterm{\svec}{b})$ or $t = (\allterm{\svec}{a})$, where the length of $\svec$ is odd. Let $M_1 = \set{*}$, and define the model $\Model_1$ with domain $M_1$ by $\semantics{a} = \semantics{b} = \set{*}$, and $\semantics{r_i} = \emptyset$ for all $i$. In $\Model_1$, we have 
\[
\semantics{\allterm{\svec}{a}} = \semantics{\allterm{\svec}{b}} = \begin{cases} 
\set{*} & \text{if the length of $\svec$ is even}\\
\emptyset & \text{if the length of $\svec$ is odd},
\end{cases}
\]
so $\Model_1\models \Gamma_n$.

\emph{Case 2:} $t = (\allterm{\svec}{b})$, where the length of $\svec$ is even. Define $\Model_2$ in the same way as $\Model_1$, but with $\semantics{b} = \emptyset$. This time, we have
\begin{align*}
\semantics{\allterm{\svec}{a}} &= \begin{cases} 
\set{*} & \text{if the length of $\svec$ is even}\\
\emptyset & \text{if the length of $\svec$ is odd}
\end{cases}\\
\semantics{\allterm{\svec}{b}} &= \begin{cases} 
\emptyset & \text{if the length of $\svec$ is even}\\
\set{*} & \text{if the length of $\svec$ is odd},
\end{cases}
\end{align*}
so again $\Model_2\models \Gamma_n$.

\emph{Case 3:} $t = a$. Fix $1\leq i\leq n$, let $M_3 = \set{*}$, and define the model $\Model_3(i)$ with domain $M_3$ by $\semantics{a} = \emptyset$, $\semantics{b} = \set{*}$, and 
\[
\semantics{r_j} =\begin{cases}\set{(*,*)} & \text{for }j\leq i\\
\emptyset & \text{for }j>i
\end{cases}
\]
In $\Model_3(i)$, we have 
\begin{align*}
\semantics{\allterm{r_j}{a}} = \semantics{b} &= \set{*}\\
\semantics{\allterm{r_j}{(\allterm{r_j}{a})}} = \semantics{\allterm{r_j}{b}} & = \begin{cases} 
\set{*} & \text{if }j\leq i\\
\emptyset & \text{if }j>i,
\end{cases}
\end{align*}
so $\Model_3(i)\models \Delta_{n,i}$.

\emph{Case 4:} $t = (\allterm{\svec}{a})$, where $\svec = (s_1,\dots,s_k)$, $k$ is even and nonzero, and $\svec\neq (r_1,r_1)$. Let $M_4 = \set{w,x,y,z}$, and define a model $\Model_4$ with domain $M_4$ by $\semantics{a} = \set{w}$ and $\semantics{b} = \set{w,x,y,z}$. To define the verb
interpretations $\semantics{r_i}$, we break into subcases.

\emph{Subcase 4a:} $s_k\neq r_1$. Set $x\semantics{r_1}w$, $y\semantics{r_1} x$, $z\semantics{s_k}w$, and no other instances of verbs:
\[
\xymatrix{
z \ar[r]^{s_k} & w & x\ar[l]_{r_1} & y\ar[l]_{r_1}
}\]

\emph{Subcase 4b:} $s_k = r_1$ and $s_{k-1}\neq r_1$. Set $x\semantics{r_1}w$, $y\semantics{r_1} x$, and no other instances of verbs.
\[
\xymatrix{
z & w & x\ar[l]_{r_1} & y\ar[l]_{r_1}
}\]

\emph{Subcase 4c:} $s_k = r_1$, $s_{k-1} = r_1$, and $k>2$. Set $x\semantics{r_1}w$, $y\semantics{r_1} x$, $z\semantics{s_{k-2}} y$, and no other instances of verbs.
\[
\xymatrix{
 w & x\ar[l]_{r_1} & y\ar[l]_{r_1} & z \ar[l]_{s_{k-2}}
}\]

In all three subcases, $\semantics{\allterm{r_1}{a}} = \set{x}$ and $\semantics{\allterm{r_1}{(\allterm{r_1}{a})}} = \set{y}$, so $\Model_4\models \alpha$. And for all $1\leq i\leq n$, $\semantics{\allterm{r_i}{b}} = \emptyset$, so $\Model_4 \models \Gamma_n$. It remains to show that $\semantics{\allterm{\svec}{a}} = \emptyset$. 
We introduce some notation: write $\svec_{\geq j}$ for the sequence $(s_j,\dots,s_k)$.

In subcase 4a, $\semantics{\allterm{s_k}{a}} = \set{z}$, and $\semantics{\allterm{s_{\geq k-1}}{a}} = \emptyset$. 
Since $k$ is even we have 
\[
\semantics{\allterm{\svec_{\geq j}}{a}} = 
\begin{cases} 
\emptyset & \text{if $j$ is odd} \\ 
M_4 & \text{if $j$ is even},
\end{cases}
\]
and $\semantics{\allterm{\svec}{a}} = \semantics{\allterm{\svec_{\geq 1}}{a}} = \emptyset$.

In subcase 4b, $\semantics{\allterm{s_k}{a}} = \set{x}$, 
and $\semantics{\allterm{s_{\geq k-1}}{a}} = \emptyset$. 
As in subcase 4a,  $\semantics{\allterm{\svec}{a}} = \emptyset$.

In subcase 4c, 
$\semantics{\allterm{\svec_{\geq k-2}}{a}} = \set{z}$.
So $\semantics{\allterm{\svec_{\geq k-3}}{a}} = \emptyset$.
And then since $k$ is even and $k\geq 4$,  $\semantics{\allterm{\svec_{\geq 1}}{a}} = \emptyset$
by the same argument which we have seen above.

This completes the proof.\end{proof}

\begin{lemma}\label{lem:modelsk}
For any natural number $k$, and any $n\geq k+1$, if $\Gamma_n\models_k \some{u}{v}$, then $u = v = (\allterm{r_1}{(\allterm{r_1}{a})})$, so that $(\some{u}{v}) = \alpha$.
\end{lemma}
\begin{proof}
By induction on the depth of the tree witnessing $\Gamma_n\models_k \some{u}{v}$. In the base case, $(\some{u}{v})$ is a leaf in $\Gamma_n$. Since $\alpha$ is the only sentence of the form $(\some{u}{v})$ in $\Gamma_n$, we are done.

Now suppose the root of the tree is $(\some{u}{v})$ with children $\set{\psi_1, \dots, \psi_j}$, where $j\leq k$, $\Gamma_n\models_k \psi_i$ for all $1\leq i\leq j$, and $\set{\psi_1, \dots, \psi_j}\models \some{u}{v}$. 

We claim that for all $\psi_i$, there is a single sentence $\chi_i\in \Gamma_n$ such that $\set{\chi_i}\models \psi_i$. By induction, if $\psi_i$ has the form $(\some{x}{y})$, then $\psi_i = \alpha$, and we can take $\chi_i = \alpha$. On the other hand, if $\psi_i$ has the form $(\all{x}{y})$, then since $\Gamma_n\models \psi_i$, by Lemma~\ref{lemma-Anti} either $\psi_i = (\all{x}{x})$ and we can take $\chi_i$ to be any sentence in $\Gamma_n$, or $\psi_i = \Anti(\rvec_i,\chi_i)$ for some $\chi_i\in \Gamma_n$, and $\set{\chi_i}\models \psi_i$.

By the claim, $\set{\chi_k :k \leq j} \models \some{u}{v}$.  
And 
since $j\leq k \leq n-1$, there is some $1\leq i^*\leq n$ such that $\set{\chi_k :k \leq j}  \subseteq \Delta_{n,i^*}$.
For this $i^*$, we see that
$\Delta_{n,i^*}\models \psi_k$ for all $k\leq j$, so $\Delta_{n,i^*}\models \some{u}{v}$. Our result follows from Lemma~\ref{lemma-r1-r1}.
\end{proof}

\begin{proof}[Proof of Theorem~\ref{theorem-stronger}]
By Lemma~\ref{lemma-some-a-a}, $\Gamma_{n+1}\models \some{a}{a}$. And by Lemma~\ref{lem:modelsk}, $\Gamma_{n+1}\not\models_n \some{a}{a}$. 

Suppose that $\proves^*$ is a sound and complete syllogistic proof system for $\langtwo$. Let $n$ be the maximum number of premises in any rule in $\proves^*$. Then since $\Gamma_{n+1}\models \some{a}{a}$, we have $\Gamma_{n+1}\proves^* \some{a}{a}$ by completeness, and $\Gamma_{n+1}\models_n \some{a}{a}$ by soundness. This is a contradiction. 
\end{proof}

We have shown that the full semantic entailment relation $\models$ for $\langtwo$ is not 
$\models_n$ for any $n$.    This contrasts with logics like the one for $\langone$; for that,
$\models$ coincides with $\models_2$, since there is a sound and complete syllogistic proof system in which every rule has at most two premises.

\begin{remark}
Our work in this section used the fact that our set $\bR$ of verbs can be arbitrarily large.
It is open whether the negative result holds when $\bR$ is a fixed finite set.
\end{remark}

\subsection{Completeness using the (\casesrule) rule}\label{sec-cases}

We have just seen that $\langtwo$ has no logical system which is  syllogistic, sound, and complete. 
The rest of this section rectifies this, in three different ways.  

In this section, we add a single rule.  It is called (\casesrule), and while
it is not \emph{syllogistic} as we defined the term in Section~\ref{section-introduction-rules},
 it is simple and natural.
Here is a statement of it.
\begin{equation}
\label{eq-cases}
\infer[\casesrule]{\phi}{
\infer*{\phi}{\xcancel{\some{x}{x}}}
&
\infer*{\phi}{\xcancel{\all{y}{(\allterm{r}{x})}}}
}
\end{equation}
The non-syllogistic feature is that premises are withdrawn in derivations.\footnote{
The reader might wonder
why we are indicating withdrawal of premises using a large ``X'' rather than the standard notation of 
square brackets.  The reason is that later in the paper we use square brackets in our syntax, and we will thus
need a different notation later to indicate withdrawals.}  
Let us explain how the (\casesrule) rule is used.    
To prove $\phi$ from a set $\Gamma$, it is sufficient to 
take a term $x$, 
prove $\phi$ from $\Gamma\cup\set{\some{x}{x}}$,
and also prove $\phi$ from $\Gamma\cup \set{\all{y}{(\allterm{r}{x})}}$.  Here $y$ can be any term and $r$ can be any verb. 

In the logic itself, we take two derivations of $\phi$,
and then in one we withdraw a sentence $(\some{x}{x})$, while in the other we withdraw a sentence of the form $(\all{y}{(\allterm{r}{x})})$ (for this same term $x$).
We may withdraw zero occurrences or more than one.  The overall conclusion is $\phi$.

In this subsection, we write $\proves$ for provability in the $\proves_0$ system from Figure~\ref{sec-basesystem},
 together with (\casesrule). 

\begin{lemma} [Soundness]  If $\Gamma\proves\phi$, then $\Gamma\models\phi$.
\label{lemma-soundness-cases}
\end{lemma}

\begin{proof}
By induction on the number $n$ of uses of (\casesrule) in derivations.   For $n = 0$, this is just soundness of $\proves_0$.
Assume our result for $n$, and let $\Gamma\proves\phi$ via a derivation with $n+1$ uses of (\casesrule).
We may assume that the last use of (\casesrule) 
is at the root of the proof tree.   So we have 
\begin{enumerate}[(1)]
\item $\Gamma\cup\set{\some{x}{x}}\proves \phi$ 
\item $\Gamma\cup\set{\all{y}{(\allterm{r}{x})}}\proves \phi$ 
\end{enumerate}
where both derivations have at most $n$ uses of $(\casesrule)$. By our induction hypotheses, (1) and (2) hold when $\proves$ is replaced by $\models$.
Let $\Model\models\Gamma$.   Then we have two cases.   If $\semantics{x} \neq \emptyset$, then $\Model\models \some{x}{x}$, so $\Model\models\phi$.
And when 
$\semantics{x} = \emptyset$, 
we have $\semantics{\allterm{r}{x}} = M$.   So $\Model\models \all{y}{(\allterm{r}{x})}$, and thus $\Model\models\phi$.
\end{proof}

\begin{example} Here is a sample derivation.
\[
 \set{\some{c}{d}, \all{a}{x}, \all{a}{y}, \all{(\allterm{r}{a})}{x}, \all{(\allterm{r}{a})}{y}}  \proves  \some{x}{y}.
 \]
Let $\Gamma$ be the theory on the left.    We show that 
\begin{enumerate}
\item $\Gamma\cup\set{\some{a}{a}}\proves \some{x}{y}$
\item
$\Gamma\cup\set{\all{c}{(\allterm{r}{a})}}\proves \some{x}{y}$
\end{enumerate}
The first is easy from $(\all{a}{x})$ and $(\all{a}{y})$. 
The second
comes from
\[
\infer[\darii]{\some{x}{y}}
{ 
\infer[\sometwo]{\some{x}{c}}{
\infer[\darii]{\some{c}{x}}
{
\infer[\someone]{\some{c}{c}}{\some{c}{d}}
&
\infer[\barbarashort]{\all{c}{x}}{ 
\all{c}{(\allterm{r}{a})}   & \all{(\allterm{r}{a})}{x} }
}}
&
\infer[\barbarashort]{\all{c}{y}}
{ \all{c}{(\allterm{r}{a})}   & {\all{(\allterm{r}{a})}{y}}}
}
\]
(Here ($\barbarashort$) abbreviates ($\barbara$).)
Note that the premise $(\all{c}{(\allterm{r}{a})})$ was used twice.
\label{example-1}
\end{example}

\begin{example}
We show that $\Gamma_n\proves \some{a}{a}$, where $\Gamma_n$ is the theory in Section~\ref{section-no-finite}. 

First, note that $\Delta\cup\set{\some{a}{a}}\proves \some{a}{a}$ for any theory $\Delta$. So by $n$ applications of $(\casesrule)$, to show that $\Gamma_n\proves \some{a}{a}$, it suffices to show that $\Gamma_n\cup \set{\all{b}{(\allterm{r_i}{a})}: 1\leq i\leq n}\proves \some{a}{a}$. Let $\Gamma_n^*$ be the theory on the left. 

Let $\psi_i = \all{(\allterm{r_i}{(\allterm{r_i}{a})})}{(\allterm{r_i}{b})}$. By $(\anti)$, $\Gamma_n^*\proves \psi_i$ for all $1\leq i\leq n$. Repeatedly applying $(\barbara)$ to the sequence $\psi_1,\phi_1,\psi_2,\phi_2,\dots, \psi_n,\omega$, we find that $\Gamma_n^*\proves \all{(\allterm{r_1}{(\allterm{r_1}{a})})}{a}$. So by $\alpha$, $(\darii)$, and $(\sometwo)$, $\Gamma_n^* \proves\some{a}{a}$. 
\end{example}

\begin{theorem} [Completeness] 
 If $\Gamma\models\phi$, then $\Gamma\proves\phi$.
 \label{theorem-completeness-cases}
\end{theorem}

\begin{proof}
By Theorem~\ref{theorem-allcomplete}, if $\Gamma\models \all{x}{y}$, then already $\Gamma\proves_0 \all{x}{y}$. So we may assume that $\varphi$ has the form $(\some{x}{y})$.  We prove the contrapositive, so assume $\Gamma\not\proves \varphi$. Let $\Terms$ be a set of terms, closed under subterms, which contains $x$, $y$, and all subterms of sentences in $\Gamma$. 
By Zorn's Lemma\footnote{Of course, we do not actually need the Axiom of Choice when the set of all sentences in the language is countable.}, let $\Gamma^*\supseteq \Gamma$ be a maximal extension, such that $\Gamma^*\not\proves \varphi$. The sentences in $\Gamma^*$ may contain any terms in the language.

Assume for contradiction that $\Gamma^*$ does not determine existentials for $\Terms$. Then there are terms $x,y\in \Terms$ and a verb $r\in \bR$ such that $\Gamma^*\not\proves_0\some{x}{x}$ and $\Gamma^*\not\proves_0 \all{y}{(\allterm{r}{x})}$. In particular, $\Gamma^*$ does not contain either of these sentences. By maximality, we have $\Gamma^*\cup \set{\some{x}{x}}\proves \varphi$ and $\Gamma^* \cup \set{\all{y}{(\allterm{r}{x})}} \proves \varphi$. By $(\casesrule)$, $\Gamma^*\proves \varphi$, contradiction.   Thus $\Gamma^*$ determines existentials for $\Terms$.

Now since $\Gamma^*\not\vdash \varphi$, we clearly have $\Gamma^*\not\vdash_0\varphi$.
So by Theorem~\ref{theorem-detextcompleteness}, $\Gamma\not\models \varphi$.
This is what was to be shown. 
\end{proof}

The proofs of Theorem~\ref{theorem-allcomplete} and Theorem~\ref{theorem-detextcompleteness} show that if $\Gamma\not\proves \varphi$, then either $\Model(\Gamma,\Terms)$ or $\Model'(\Gamma^*,\Terms)$ are countermodels, depending on the form of $\varphi$. Both of these models have size $O(n^2)$, where $n$ is the complexity of $\Gamma\cup \set{\varphi}$.

\begin{remark}
Since $\proves_0$ is already complete for sentences of the form $(\all{a}{b})$, we only need to use (\casesrule) in proofs of sentences of the form $(\some{a}{b})$. Using Lemma~\ref{lemma-someproof}, it is possible to show that (\casesrule) is equivalent over $\proves_0$ to the following rule: 
\[
\infer[\casesrule^*]{\some{a}{b}}{
\all{x}{a} & \all{x}{b} &
\infer*{\some{a}{b}}{\xcancel{\all{y}{(\allterm{r}{x})}}}}
\]

So the system with rules (\axiom), (\barbara), (\anti), (\someone), (\sometwo), (\darii), and ($\casesrule^*$) is also sound and complete for $\langtwo$. We chose to emphasize (\casesrule) rather than $(\casesrule^*)$ because it made the completeness proof quicker, and because we will use (\casesrule) again in Section~\ref{section-completeness-langthreesecond}. 
\end{remark}

\subsection{Completeness using the (\chains) schema}
\label{section-chains}

In Theorem~\ref{theorem-stronger}, 
we proved that there are no syllogistic proof systems for $\langtwo$ which are  sound
and complete.
We have just seen that $\langtwo$ has a sound and complete proof system, but one which is not syllogistic in our sense.
In this section, we give another proof system, which this time  makes use  of a schema of rules with arbitrarily long (but finite) premise sets.

\begin{definition} Let $a$ and $b$ be nouns.
A \emph{chain linking $a$ to $b$} is a sequence $C$ of sentences
\[
C = (\all{a}{u_1}, \all{v_1}{u_2},\dots,\all{v_i}{u_{i+1}},\dots,\all{v_m}{b}),
 \]
 such that for all $1\leq i \leq m$, either
\begin{enumerate}
\item $u_i = (\allterm{\rvec}{z_i})$ and $v_{i} = (\allterm{\rvec}{(\allterm{r}{t_i})})$, where $\rvec$ is a sequence of even length, or  
\item $u_i = (\allterm{\rvec}{(\allterm{r}{t_i})})$ and $v_{i} = (\allterm{\rvec}{z_i})$, where $\rvec$ is a sequence of odd length.
\end{enumerate}
We say that this chain has length $(m+1)$, and the terms $t_1,\dots,t_m$ are the \emph{missing link terms} in $C$.
\label{def-T-chain}
\end{definition}

Note that a chain of length $1$ linking $a$ to $b$ is just the single sentence $(\all{a}{b})$ and has no missing link terms. 
Here are two chains of length $2$ linking $a$ to $b$: 
\begin{gather*}
(\all{a}{z}, \all{(\allterm{r}{t})}{b})\\
(\all{a}{(\allterm{s}{(\allterm{r}{t})})}, \all{(\allterm{s}{z})}{b})
\end{gather*}
In both of these chains, $t$ is the missing link term.

Returning to the definition, we emphasize that 
the terms denoted $z_1,\dots,z_m$  may be arbitrary (they need not be nouns) and are not missing link terms.
The sequence $\rvec$ is also arbitrary.

\begin{lemma} Suppose $C = (\all{a}{u_1}, \all{v_1}{u_2},\dots,\all{v_m}{b})$ is a chain linking $a$ to $b$, $\Model\models C$, and $\semantics{t} = \emptyset$ for every missing link term $t$ in $C$. Then $\semantics{a}\subseteq \semantics{b}$.
\label{lemma-one-chain}
\end{lemma}
\begin{proof}
Since $\Model$ satisfies all the sentences in $C$, $\semantics{a}\subseteq \semantics{u_1}$, $\semantics{v_m}\subseteq \semantics{b}$, and  $\semantics{v_i}\subseteq \semantics{u_{i+1}}$ for all $i$. So it suffices to show that $\semantics{u_i}\subseteq \semantics{v_i}$ for all $i$. 

Let $t_i$ be the missing link term for $u_i$ and $v_i$. Since $\semantics{t_i} = \emptyset$, $\semantics{\allterm{r}{t_i}} = \Model$.
 So $\semantics{z_i}\subseteq \semantics{\allterm{r}{t_i}}$ for any term $z_i$. This is the desired 
inclusion when  $\rvec$ is the empty sequence. The result then follows by induction on the length of $\rvec$, using the fact that if $\semantics{x}\subseteq \semantics{y}$, then $\semantics{\allterm{r}{y}}\subseteq \semantics{\allterm{r}{x}}$.
\end{proof}

\begin{definition}
Let $x$ and $y$ be terms. An \emph{$(x,y)$ chain system} is a sequence of chains $C_1,\dots,C_l$ such that 
for every missing link term $t$ in every chain $C_n$, there exist $m,m'<n$ such that $C_m$ links $t$ to $x$ and $C_{m'}$ links $t$ to $y$. When $x = y$, we may take $m = m'$. 
\end{definition}

If chains $C_m$ and $C_{m'}$ link $t$ to $x$ and $y$, respectively, we can think of these chains as witnessing that we are allowed to use $t$ as a missing link term later in the $(x,y)$ chain system. Of course, the first chain in an $(x,y)$ chain system must have length $1$, since there are no available missing link terms from previous chains.

\begin{lemma} Let $C_1,\dots,C_l$ be an $(x,y)$ chain system. Suppose $\Model\models C_i$ for all $i$, and suppose that $\semantics{x}\cap \semantics{y} = \emptyset$ in $\Model$. Then $\semantics{a}\subseteq \semantics {b}$ whenever some $C_i$ links $a$ to $b$.
\label{l33}
\end{lemma} 
\begin{proof}
By induction on $l$. When $l = 0$, there are no chains, so the conclusion is vacuously satisfied. Now suppose $C_1,\dots,C_{l+1}$ is an $(x,y)$ chain system. By induction, the conclusion holds for the $(x,y)$ chain system $C_1,\dots,C_l$. So suppose $C_{l+1}$ links $a$ to $b$. For any missing link term $t$ in $C_{l+1}$, there are chains $C_i$ and $C_j$ with $i,j\leq l$, linking $t$ to $x$ and $t$ to $y$. So $\semantics{t}\subseteq \semantics{x}\cap \semantics{y} = \emptyset$. By Lemma~\ref{lemma-one-chain}, $\semantics{a}\subseteq \semantics{b}$, as desired.
\end{proof}

We introduce the new rule schema
\[
\infer[\mbox{\chains}]{\some{x}{y}}{
\some{a}{b} & C_1 & \dots & C_l}
\]
where $C_1,\dots,C_l$ is an $(x,y)$ chain system, some $C_m$ links $a$ to $x$, and some $C_n$ links $b$ to $y$.

In this section, we write $\proves$ for the proof system $\proves_0$ augmented by the rule schema (\chains).

\begin{theorem} The $(\chains)$ schema is sound.
\end{theorem}
\begin{proof} 
Let $(\some{a}{b}), C_1,\dots,C_l$ be the premises of an instance of (\chains), and suppose   that $\Model$ satisfies these premises.
Suppose towards a contradiction that $\semantics{x} \cap\semantics{y} = \emptyset$, so Lemma~\ref{l33} applies.
 Since some $C_m$ links $a$ to $x$ and some $C_n$ links $b$ to $y$, we 
use Lemma~\ref{l33} to see that $\semantics{a}\cap \semantics{b}\subseteq \semantics{x}\cap\semantics{y} = \emptyset$.
But this contradicts the assumption that $\Model\models \some{a}{b}$.
\end{proof}

\begin{example}
We show that $\Gamma_n\proves \some{a}{a}$, where $\Gamma_n$ is the theory in Section~\ref{section-no-finite}.  We will find an $(a,a)$ chain system which contains a chain linking $(\allterm{r_1}{(\allterm{r_1}{a})})$ to $a$. 

$C_1 = (\all{a}{a})$  is a chain of length $1$ linking $a$ to $a$. This allows $a$ to be used as a missing link term in $C_2$. Let $\beta$ be the sentence
\[
\all{(\allterm{r_1}{(\allterm{r_1}{a})})}{(\allterm{r_1}{(\allterm{r_1}{a})})}.
\]
Then 
\[
C_2 = (\beta, \phi_1, \ldots, \phi_{n-1}, \omega)
\]
is a chain linking $(\allterm{r_1}{(\allterm{r_1}{a})})$ to $a$, in which the only missing link term is $a$.

Let's check that $C_2$ is a chain. First, $u_1  = (\allterm{r_1}{(\allterm{r_1}{a})})$ and  $v_1 =  (\allterm{r_1}{c})$, so this is alternative $2$ in Definition~\ref{def-T-chain}, with 
$t = a$, $z = c$, $r = r_1$, and $\rvec = r_1$. All of the rest of the links from $u_i$ to $v_i$ are justified in the same way.

Then we have a derivation:
\[
\infer[\chains]{\some{a}{a}}{\alpha & C_1 & C_2}
\]
This shows that $\Gamma_n\proves \some{a}{a}$,
because $\alpha\in \Gamma_n$, and each sentence in $C_1, C_2$ is either $\beta$, which is an instance of $(\axiom)$,
or an element of $\Gamma_n$.
\label{example-gamma-n-chains}
\end{example}

\begin{theorem}[Completeness]\label{theorem-chains}
If $\Gamma\models\phi$, then $\Gamma\proves \phi$.
\end{theorem}
\begin{proof}
By Theorem~\ref{theorem-allcomplete}, if $\Gamma\models \all{x}{y}$, then already $\Gamma\proves_0 \all{x}{y}$. So we may assume that $\varphi$ has the form $(\some{x}{y})$. Let $\Terms$ be a set of terms, closed under subterms, which contains $x$, $y$, and all subterms of sentences in $\Gamma$.

For any term $t\in \Terms$, let $\Delta_t = \set{\all{z}{(\allterm{r}{t})} :  z \in \Terms, r \in \bR}$. In order to extend $\Gamma$ to a theory which determines existentials for $\Terms$, we define an increasing sequence of theories by induction. Set $\Gamma_0 = \Gamma$, and given $\Gamma_n$, define 
\[
\Gamma_{n+1} = \Gamma_n\cup \set{\all{z}{(\allterm{r}{t})} : t,z\in \Terms, r\in \bR, \text{and }\Gamma_n\cup \set{\some{t}{t}}\proves_0 \varphi}.
\]
So $\Gamma_{n+1}$ includes $\Delta_t$ for all terms $t\in \Terms$ such that $\Gamma_n\cup \set{\some{t}{t}}\proves_0 \varphi$. 
Let $\Gamma_\omega = \bigcup_{n\in \omega} \Gamma_n$, and let 
\[\Gamma^* = \Gamma_\omega\cup \set{\some{t}{t} : t\in \Terms\text{ and }\Delta_t\not\subseteq \Gamma_\omega}.\]

\begin{claim}\label{claim-chains}
There is some $n$ such that $\Gamma_n\proves_0\varphi$. 
\end{claim}

\begin{proof}
 $\Gamma^*$ determines existentials for $\Terms$,
and by Theorem~\ref{theorem-detextcompleteness}, $\Gamma^*\proves_0 \varphi$. 
By Lemma~\ref{lemma-someproof}, there is a sentence  $(\some{a_1}{a_2})\in \Gamma^*$  such that $\Gamma^*\proves_0 \all{a_i}{x}$ and $\Gamma^*\proves_0\all{a_j}{y}$ for some $i,j\in \set{1,2}$. Since $\proves_0$-proofs of {\sf all}-sentences are finite and never contain {\sf some}-sentences, there is already some $n$ such that $\Gamma_n\proves_0 \all{a_i}{x}$ and $\Gamma_n\proves_0\all{a_j}{y}$, and we have $\Gamma_n\cup \set{\some{a_1}{a_2}}\proves_0 \varphi$. 

It remains to show that the sentence $(\some{a_1}{a_2})$ belongs to $\Gamma\subseteq \Gamma_n$, since then $\Gamma_n\proves_0 \varphi$. Suppose not.   Then 
 $(\some{a_1}{a_2})$
  is not in $\Gamma_\omega$, since the sets $\Gamma_n$ only add sentences of the form $(\all{z}{(\allterm{r}{t})})$ to $\Gamma$. 
  So  $(\some{a_1}{a_2})$
   is a sentence $(\some{t}{t})$ such that $\Delta_t\not\subseteq \Gamma_\omega$. But then $\Gamma_n\cup \set{\some{t}{t}}\proves_0 \varphi$, so $\Delta_t\subseteq \Gamma_{n+1}\subseteq \Gamma_\omega$, which is a  contradiction.
\end{proof}

\begin{claim}\label{claim-hard-chains}
If there is some $n$ such that $\Gamma_n\proves_0\phi$, then $\Gamma\proves \phi$.
\end{claim}
\begin{proof}
Assume $\Gamma\not\proves \phi$. Then we will prove the following two claims for all $n$, by induction: 
\begin{enumerate}[(1)$_n$]
\item For every $k\geq 1$ and every
$\Gamma_{n}$-sequence of terms $t_1,\dots,t_k$, 
 there is an $(x,y)$ chain system $C_1,\dots,C_\ell$, such that
$C_\ell$ links $t_1$ to $t_k$, and such that for all $1\leq i\leq \ell$ and all $\psi\in C_i$,
$\Gamma\proves \psi$. 
\item $\Gamma_n\not\proves_0\phi$. 
\end{enumerate}

So assume (1)$_m$ and (2)$_m$ hold for all $m<n$. We will first prove (1)$_n$ by induction on $k$. 

In the base case, when $k = 1$, we have  $t_1 = t_k$, and $\Gamma  \proves \all{t_1}{t_k}$ by  (\axiom). 
The chain $(\all{t_1}{t_k})$ has no missing link terms and links $t_1$ to $t_k$. So we have the required $(x,y)$ chain system, consisting of just this one chain. 

Now suppose $k > 1$, and fix a $\Gamma_{n}$-sequence $t_1,\dots,t_k$.
Let $C_1,\dots,C_\ell$ be the $(x,y)$ chain system obtained by induction for the $\Gamma_n$-sequence $t_1,\dots, t_{k-1}$. Then $C_\ell$ links $t_1$ to $t_{k-1}$, so the last sentence in $C_{l}$ is $(\all{c}{t_{k-1}})$ for some term $c$. By the definition of $\Gamma_n$-sequence, there is a sentence $\all{d}{e}\in \Gamma_n$ and a sequence of verbs $\rvec$ such that $(\all{t_{k-1}}{t_k}) = \Anti(\rvec,(\all{d}{e}))$.

If $\all{d}{e}\in \Gamma$, then by repeated applications of $(\anti)$, $\Gamma\proves \all{t_{k-1}}{t_k}$. Since also $\Gamma\proves\all{c}{t_{k-1}}$ by induction, we have $\Gamma\proves \all{c}{t_k}$ by $(\barbara)$. Replacing the last sentence of $C_\ell$ with $(\all{c}{t_k})$, we are done.

If $\all{d}{e}\notin \Gamma$, then $(\all{d}{e})\in \Gamma_{m+1}\setminus \Gamma_{m}$ for some $0\leq m < n$. It follows that $e = (\allterm{r}{t})$ for some term $t$ such that $\Gamma_{m}\cup \set{\some{t}{t}}\proves_0 \phi$.  
By Lemma~\ref{lemma-someproof}, there is a sentence $(\some{p}{q})$ in $\Gamma_{m}\cup \set{\some{t}{t}}$ such that
for all $w\in\set{x,y}$ there is some $v\in\set{p,q}$ such that $\Gamma_{m}\proves_0 \all{v}{w}$.  
If $\some{p}{q}\in\Gamma_{m}$, then $\Gamma_{m}\proves_0 \phi$, contradicting (2)$_m$. 
Otherwise, 
 the sentence $(\some{p}{q})$ must be $(\some{t}{t})$. That is, $p = q = t$, so $\Gamma_{m}\proves_0 \all{t}{x}$ and $\Gamma_{m}\proves_0 \all{t}{y}$.

 By Lemma~\ref{lemma-allproof}, there are $\Gamma_{m}$-sequences linking $t$ to $x$ and $t$ to $y$. By (1)$_m$, there are $(x,y)$ chain systems $C'_1,\dots,C'_{\ell'}$ and $C''_1,\dots,C''_{\ell''}$ such that $C'_{\ell'}$ links $t$ to $x$, $C''_{\ell''}$ links $t$ to $y$, and for every sentence $\psi$ in every chain, $\Gamma\proves \psi$. Let $C_\ell^*$ be the chain $C_\ell$ with the sentence $(\all{t_k}{t_k})$ appended. Then $C_\ell^*$ links $t_1$ to $t_k$. We either have $t_{k-1} = (\allterm{\rvec}{d})$ and $t_k = (\allterm{\rvec}{(\allterm{r}{t})})$ where $\rvec$ has even length, or $t_{k-1} = (\allterm{\rvec}{(\allterm{r}{t})})$ and $t_k = (\allterm{\rvec}{d})$ where $\rvec$ has odd length, so the missing link terms in $C_\ell^*$ are those in $C_\ell$, together with $t$. Also $\Gamma\proves \psi$ for every sentence $\psi$ in $C^*_\ell$, by our assumption about $C_\ell$ and (\axiom). So $C'_1,\dots,C'_{\ell'},C''_1,\dots,C''_{\ell''},C_1,\dots,C_{\ell-1},C_\ell^*$ is our desired $(x,y)$ chain system. 

Having established (1)$_n$, we prove (2)$_n$. Assume for contradiction that $\Gamma_n\proves_0\phi$.  By Lemma~\ref{lemma-someproof}, there is a sentence ($\some{a_1}{a_2}$) in $\Gamma_n$ such that
$\Gamma_n\proves_0 \all{a_i}{x}$ and $\Gamma_n\proves_0 \all{a_j}{y}$ for some $i,j\in \set{1,2}$.
By Lemma~\ref{lemma-someproof}, we thus have $\Gamma_n$-sequences from $a_i$ to $x$ and from $a_j$ to $y$.
Apply (1)$_n$  to these sequences to obtain $(x,y)$ chain systems $C_1,\dots,C_\ell$ and $C'_1,\dots,C'_{\ell'}$  such that $C_l$ links $a_i$ to $x$ and $C'_{\ell'}$ links $a_j$ to $y$, and all sentences in all chains are $\proves$-provable from $\Gamma$. Then we have an instance of chains 
\[
\infer[\chains]{\some{x}{y}}{\some{a_i}{a_j} & C_1 & \dots & C_\ell & C_1' & \dots & C_\ell'}
\]
The sentence $(\some{a_1}{a_2})$ belongs to $\Gamma_n$, hence to  $\Gamma$.
 Applying $(\someone)$ or $(\sometwo)$ as needed, $\Gamma\proves \some{a_i}{a_j}$. So this deduction is the root of a proof tree proving $(\some{x}{y})$ from $\Gamma$.
\end{proof}
 
Claims~\ref{claim-chains} and~\ref{claim-hard-chains} complete the proof of Theorem~\ref{theorem-chains}.
\end{proof}

\subsection{Completeness and $\PTIME$ decidability for the extended language $\langtwo^+$}
\label{section-fourplace}

We have seen that  $\langtwo$ has no sound and complete syllogistic proof system.  
And we have seen proof systems which go beyond the ``purely syllogistic'' in earlier sections.    But this section goes in a different direction.
We show that if we enhance the syntax
of $\langtwo$ in a certain way, then we \emph{are} able to find a boundedly complete syllogistic proof system for the larger language.

We add to $\langtwo$ a new four-place sentence former 
\[
\allorsome{a}{b}{x}{y}
\]
with the evident semantics
\[
\Model\models \allorsome{a}{b}{x}{y}
\quadiff
\semantics{a}\subseteq \semantics{b} \mbox{ or } \semantics{x} \cap \semantics{y} \neq \emptyset.
\]
We call the larger language $\langtwo^+$. Note that we do not allow the disjunction of arbitrary sentences. Rather, there is a new kind of sentence, which is the disjunction of exactly one sentence $(\all{a}{b})$ and one sentence $(\some{x}{y})$.
For a proof system, we take the rules in Figure~\ref{Fig-fourplace-rules}, and we write $\proves$ for provability in this system.

\begin{figure}[h]
\begin{mathframe}
\begin{gather*}
\infer[\emptyone]{\allorsome{a}{b}{a}{a}}{} \qquad
\infer[\newsomeone]{\allorsome{a}{b}{x}{x}}{\allorsome{a}{b}{x}{y}} \qquad
\infer[\lweak]{\allorsome{a}{b}{x}{y}}{\all{a}{b}} \\\\
\infer[\emptytwo]{\allorsome{b}{(\allterm{r}{a})}{a}{a}}{}\qquad
\infer[\newsometwo]{\allorsome{a}{b}{y}{x}}{\allorsome{a}{b}{x}{y}}\qquad
 \infer[\rweak]{\allorsome{a}{b}{x}{y}}{\some{x}{y}}\\\\
\infer[\newbarbara]{\allorsome{a}{c}{x}{y}}{\allorsome{a}{b}{x}{y} & \allorsome{b}{c}{x}{y}}\qquad
\infer[\newanti]{\allorsome{(\allterm{r}{b})}{(\allterm{r}{a})}{x}{y}}{\allorsome{a}{b}{x}{y}}\\\\
\infer[\newdarii]{\some{x}{y}}{\some{t}{u} & \allorsome{t}{x}{x}{y} & \allorsome{u}{y}{x}{y}}\\\\
\infer[\newnewdarii]{\allorsome{a}{b}{x}{y}}{\allorsome{a}{b}{t}{u} & \allorsome{t}{x}{x}{y} & \allorsome{u}{y}{x}{y}}
\end{gather*}
\caption{Rules in Section~\ref{section-fourplace}.
These rules are added on top of the rules in Figure~\ref{baserules}.\label{Fig-fourplace-rules}}
\end{mathframe}
\end{figure}

\begin{lemma} The proof system is sound.
\end{lemma}
\begin{proof}
We argue soundness for $(\emptyone)$, $(\emptytwo)$, and $(\newdarii)$, since soundness of the other rules is clear from the meaning of disjunction. Fix a model $\Model$.

For $(\emptyone)$, either $\semantics{a}\neq \emptyset$, or $\emptyset = \semantics{a}\subseteq \semantics{b}$. In either case, $\Model\models \allorsome{a}{b}{a}{a}$. 

For $(\emptytwo)$, either $\semantics{a} \neq \emptyset$, or $\semantics{a} = \emptyset$. In the latter case, $\semantics{b}\subseteq \semantics{\allterm{r}{a}} = M$. In either case, $\Model\models \allorsome{b}{(\allterm{r}{a})}{a}{a}$. 

For $(\newdarii)$, suppose that $\Model$ satisfies the premises of the rule, and assume for contradiction that $\semantics{x}\cap \semantics{y} = \emptyset$. Then $\Model\models \all{t}{x}$ and $\Model\models \all{u}{y}$. So $\semantics{t}\cap\semantics{u}\subseteq \semantics{x}\cap \semantics{y} = \emptyset$, contradicting $\Model\models \some{t}{u}$.  
\end{proof}

\begin{example}
The rules $(\newdarii)$ and its companion $(\newnewdarii)$ stand for ``double darii''. To see the connection between these rules and $(\darii)$, we note that $(\darii)$ is redundant in this system:
\[
\infer[\newdarii]{\some{x}{z}}{\some{x}{y} & \infer[\lweak]{\allorsome{x}{x}{x}{z}}{\infer[\axiom]{\all{x}{x}}{}} & \infer[\lweak]{\allorsome{y}{z}{x}{z}}{\all{y}{z}}}
\]
Note also that $(\newdarii)$ is the only new rule in our proof system which produces a conclusion in the original syntax of $\langtwo$. It is responsible, together with $(\emptytwo)$, for the reasoning captured by $(\casesrule)$ and $(\chains)$ in the previous two sections.
\end{example}

\begin{example}  
We show that $\Gamma_n\proves \some{a}{a}$, where once again 
$\Gamma_n$ is from Section~\ref{section-no-finite}. 
For each $1\leq i < n$, we have the derivation:
\[
\infer[\newbarbara]{\allorsome{(\allterm{r_i^2}{a})}{(\allterm{r_{i+1}^2}{a})}{a}{a}}
{\infer[\newanti]{\allorsome{(\allterm{r_i^2}{a})}{(\allterm{r_i}{b})}{a}{a}}
{\infer[\emptytwo]{\allorsome{b}{(\allterm{r_i}{a})}{a}{a}}{}}
&
\infer[\lweak]{\allorsome{(\allterm{r_i}{b})}{(\allterm{r_{i+1}^2}{a})}{a}{a}}{\phi_i}
}
\]
and similarly, we have:
\[
\infer[\newbarbara]{\allorsome{(\allterm{r_n^2}{a})}{a}{a}{a}}
{\infer[\newanti]{\allorsome{(\allterm{r_n^2}{a})}{(\allterm{r_n}{b})}{a}{a}}
{\infer[\emptytwo]{\allorsome{b}{(\allterm{r_n}{a})}{a}{a}}{}}
&
\infer[\lweak]{\allorsome{(\allterm{r_n}{b})}{a}{a}{a}}{\omega}
}
\]
By $n$ applications of $(\newbarbara)$, we obtain $\allorsome{(\allterm{r_1^2}{a})}{a}{a}{a}$. And then we conclude:
\[
\infer[\newdarii]{\some{a}{a}}{
\alpha & \allorsome{(\allterm{r_1^2}{a})}{a}{a}{a}
}
\]
\end{example}

\paragraph{Completeness and $\PTIME$-decidability}
At this point, we turn to the completeness and $\PTIME$-decidability of the logic. 
We are going to apply Theorem~\ref{theorem-ptime}. 
For any set $\Delta$,
let $T(\Delta)$ be the set of subterms of sentences in $\Delta$.
Let $T^+(\Delta)$ be $T(\Delta)$ together with the terms $(\allterm{r}{w})$
where $w\in T(\Delta)$ and where $r$ occurs in $\Delta$.

Let $g(\Delta)$ be the set consisting of 
\begin{description}
\item{(i)} All sentences $(\all{x}{y})$,
where $x\in T(\Delta)$ and $y\in T^+(\Delta)$
\item{(ii)} All sentences $(\some{u}{v})$, where $u, v\in T(\Delta)$.
\item{(iii)} All sentences $\allorsome{x}{y}{u}{v}$, where $x,u,v\in T(\Delta)$ and $y\in T^+(\Delta)$.
\end{description}
Note that $g$ is computable in $\PTIME$.

\begin{theorem}\label{theorem-twoplus} 
If  $\Gamma\models \phi$, then $\Gamma\proves_{g(\Gamma\cup\set{\phi})} \phi$.
Hence the consequence relation for $\langtwo^+$ is in $\PTIME$.
\end{theorem}

\begin{proof}
As in Section~\ref{section-strongly},
we are going to save on some notation below by writing $T$ for $T(\Gamma\cup\set{\phi})$,
$T^+$ for $T^+(\Gamma\cup\set{\phi})$,
and $A$ for $g(\Gamma\cup\set{\phi})$.

We are going to do this entire proof in two parts.   The first part handles the 
case that $\phi$ is either  $\allorsome{a}{b}{x}{y}$ or else $(\some{x}{y})$.  
After that, we handle the relatively simpler case that $\phi$ is $(\all{a}{b})$.

So until further notice, we are in the first part of this theorem.
Please note that $x$, $y$, $a$, and $b$ are fixed throughout the rest of this proof.

Let $M = M_{xy}$ be the set of unordered pairs $\set{t,u}$ of terms from $T$ such that there is some $z\in \set{x,y}$ such that for all $v\in \set{t,u}$, $\Gamma\not\provesA \allorsome{v}{z}{x}{y}$.

We allow $t= u$, and it follows that
whenever  $M$ contains $\set{t,u}$, then it also contains 
   $\set{t} = \set{t,t}$.
 The point of the definition of $M$ will become clearer after we see the Truth Lemma and Claim~\ref{lastclaimlangtwoplus} below:
 we are building a model which is guaranteed to have $\semantics{x}\cap\semantics{y} = \emptyset$.  
   
    We define a model $\Model$ with domain $M$ by setting
\begin{equation*}
\begin{split}
\set{t,u}\in \semantics{p} &\quadiff  \mbox{either } \Gamma \provesA \allorsome{t}{p}{x}{y} \mbox{, or } \Gamma \provesA \allorsome{u}{p}{x}{y}\\
\set{t,u}\semantics{r}\set{v,w}  & \quadiff \mbox{for some $c\in \set{t,u}$ and $d\in\set{v,w}$, }\Gamma\provesA\allorsome{c}{(\allterm{r}{d})}{x}{y}.
\end{split}
\end{equation*}

\begin{claim} [Truth Lemma] In $\Model$,we have the following for all terms $z\in T$,
\[
\semantics{z} =  \set{\set{t,u}\in M :  \Gamma\provesA \allorsome{t}{z}{x}{y} \mbox{ or }\Gamma \provesA \allorsome{u}{z}{x}{y}}.
\]\label{langtwoplustruthlemma}
\end{claim}
\begin{proof}
By induction on $z$. For a noun in $\bP$, this is by definition of the model.
So we assume our statement for $z$ and prove it for $(\allterm{r}{z})\in T$.
Note that $z\in T$, since $T$ is closed under subterms.   

Fix $\set{t,u}\in M$, and suppose (without loss of generality) $\Gamma\provesA\allorsome{t}{(\allterm{r}{z})}{x}{y}$. 
We show that 
 $\set{t,u}\in \semantics{\allterm{r}{z}}$. 
 Let $\set{v,w}\in M$ be an element of $\semantics{z}$. By the induction hypothesis, we have (without loss of generality) $\Gamma\provesA\allorsome{v}{z}{x}{y}$. Then by $(\newanti)$, $\Gamma\provesA \allorsome{(\allterm{r}{z})}{(\allterm{r}{v})}{x}{y}$.
Note that the sentence $\allorsome{(\allterm{r}{z})}{(\allterm{r}{v})}{x}{y}$ belongs to $A$
because $\allterm{r}{z}$, $x$, and $y$ belong to $T$ and 
$\allterm{r}{v}$ to $T^+$.  Moreover,  $\allorsome{t}{(\allterm{r}{v})}{x}{y}$ again belongs to $A$.
By $(\newbarbara)$, $\Gamma\provesA \allorsome{t}{(\allterm{r}{v})}{x}{y}$, and hence $\set{t,u}\semantics{r} \set{v,w}$. So $\set{t,u}\in \semantics{\allterm{r}{z}}$. 
 
Conversely, suppose $\set{t,u}\in\semantics{\mbox{\sf $r$ all $z$}}$. 

Case 1: $\set{z}\in M$.  Notice that $\Gamma\provesA\allorsome{z}{z}{x}{y}$ by $(\axiom)$ and $(\lweak)$. 
By induction, $\set{z}\in \semantics{z}$.  Hence $\set{t,u}\semantics{r}\set{z}$.  So  by the definition of the model, we have
 either $\Gamma\provesA \allorsome{t}{(\allterm{r}{z})}{x}{y}$ or $\Gamma \provesA \allorsome{u}{(\allterm{r}{z})}{x}{y}$, as desired.

Case 2: $\set{z}\notin M$. Then $\Gamma\provesA\allorsome{z}{x}{x}{y}$ 
and $\Gamma\provesA\allorsome{z}{y}{x}{y}$. So we have a proof from $\Gamma$:
\[
\infer[\newnewdarii]{\allorsome{t}{(\allterm{r}{z})}{x}{y}}{\infer[\emptytwo]{\allorsome{t}{(\allterm{r}{z})}{z}{z}}{} &
\infer*{\allorsome{z}{x}{x}{y}}{} & \infer*{\allorsome{z}{y}{x}{y}}{}}
\]
As before, all sentences shown belong to $A$.
We are done.
Incidentally,
the same argument shows that also $\Gamma\provesA\allorsome{u}{(\allterm{r}{z})}{x}{y}$. 
\end{proof}

We conclude the first part of our proof of Theorem~\ref{theorem-twoplus} 
with two claims.  Together with the assumption that $\Gamma\models\phi$, they show that $\Gamma\provesA \phi$, where 
$\phi$ is the sentence in the statement of our theorem.   
In this part of the proof, recall that $\phi$ is
either $\allorsome{a}{b}{x}{y}$ or $(\some{x}{y})$.

\begin{claim}
\label{nexttolastclaimlangtwoplus}
 Either $\Gamma\provesA \phi$, or 
$\Model \models\Gamma$.
\end{claim}

\begin{proof}  
Let $\psi$ be a sentence in $\Gamma$.   We check that either $\Gamma\provesA \phi$, or 
$\Model \models\psi$.

Case 1: $\psi$ is  $(\all{c}{d})$. By $(\lweak)$, $\Gamma\provesA \allorsome{c}{d}{x}{y}$. For any $\set{t,u}\in \semantics{c}$, by the Truth Lemma (without loss of generality) $\Gamma\provesA \allorsome{t}{c}{x}{y}$. By $(\newbarbara)$, $\Gamma\provesA\allorsome{t}{d}{x}{y}$.
So $\set{t,u}\in \semantics{d}$ by the Truth Lemma again. So in this case, we have $\Model\models \psi$. 

Case 2: $\psi$ is  $(\some{c}{d})$. 
There are two subcases, depending on whether  or not $\set{c,d}$ belongs to  $M$.
If it does, then  $\set{c,d}\in \semantics{c}\cap \semantics{d}$, so $\Model\models \psi$. 
So we assume that $\set{c,d}\notin M$. 
Then there are  $e,f\in \set{c,d}$ such that $\Gamma\provesA \allorsome{e}{x}{x}{y}$ and $\Gamma\provesA\allorsome{f}{y}{x}{y}$. Applying $(\someone)$ and $(\sometwo)$ as needed, $\Gamma\provesA \some{e}{f}$. Then by $(\newdarii)$, $\Gamma\provesA \some{x}{y}$. 
If $\phi$ is $(\some{x}{y})$, then we are immediately done.
 And if $\phi$ is $\allorsome{a}{b}{x}{y}$, then we are done after applying $(\rweak)$. 

Case 3: $\psi$ is  $\allorsome{s}{t}{c}{d}$. This is a combination of the two previous arguments. 

Again there are two subcases, depending on whether  or not $\set{c,d}$ belongs to  $M$.
If it does, then  $\set{c,d}\in \semantics{c}\cap \semantics{d}$, so $\Model\models \psi$. 
So we assume that $\set{c,d}\notin M$. 
Then there are  $e,f\in \set{c,d}$ such that $\Gamma\provesA \allorsome{e}{x}{x}{y}$ and $\Gamma\provesA\allorsome{f}{y}{x}{y}$. Applying $(\newsomeone)$ and $(\newsometwo)$ as needed, $\Gamma\provesA \allorsome{s}{t}{e}{f}$. Then by $(\newnewdarii)$, $\Gamma\provesA \allorsome{s}{t}{x}{y}$. For any $\set{u,v}\in \semantics{s}$, by the Truth Lemma (without loss of generality) $\Gamma\provesA \allorsome{u}{s}{x}{y}$. By $(\newbarbara)$, $\Gamma\provesA\allorsome{u}{t}{x}{y}$.
So $\set{u,v}\in \semantics{t}$ by the Truth Lemma again. So we have $\Model\models \all{s}{t}$, and $\Model\models \psi$.  
\end{proof}

\begin{claim}
\label{lastclaimlangtwoplus}
If $\Model \models\phi$,  then $\Gamma\provesA \phi$. 
\end{claim}

\begin{proof}  
Case 1: $\phi$ is $(\some{x}{y})$.
In this case, we claim that $\Model \not\models\phi$.
The reason is that by the Truth Lemma and the definition of $M$,
$$\set{t,u}\in M \quad\mbox{iff}\quad \set{t,u}\notin\semantics{x}\cap\semantics{y}.
$$

 Case 2: $\phi$ is $\allorsome{a}{b}{x}{y}$.
By Case 1, we assume that $\Model\models\all{a}{b}$.
Consider $\set{a}$. If $\set{a}\notin M$, then $\Gamma\provesA\allorsome{a}{x}{x}{y}$ and $\Gamma\provesA\allorsome{a}{y}{x}{y}$. Then we have the following proof of $\phi$ from $\Gamma$:
\[
\infer[\newnewdarii]{\allorsome{a}{b}{x}{y}}{
\infer[\emptyone]{\allorsome{a}{b}{a}{a}}{} &
\infer*{\allorsome{a}{x}{x}{y}}{} & 
\infer*{\allorsome{a}{y}{x}{y}}{}
}
\]
 All sentences shown belong to $A$.
So we have the desired conclusion  $\Gamma\provesA \phi$. 
 On the other hand, if $\set{a}\in M$, then since 
 $\set{a}\in\semantics{a}$ and 
 $\Model\models\all{a}{b}$, we have $\set{a}\in\semantics{b}$.
By the Truth Lemma, we again have $\Gamma\provesA \allorsome{a}{b}{x}{y}$.
\end{proof}

This concludes the first part of the proof of Theorem~\ref{theorem-twoplus}.
The second part is when $\phi$ is a sentence $(\all{c}{d})$.   In this part, we repeat the construction and proof above, with the following adjustments: 
\begin{enumerate}
\item We let $M$ be the set of all unordered pairs of terms from $T$, with no restriction.
\item We drop the disjunct $\lor (\some{x}{y})$ from all sentences which appear in the proof, including the definition of $\Model$ and the statement of the Truth Lemma.
\item In the proof of the Truth Lemma, Case 2 does not occur, since $\set{z}\in M$. 
\item In the proof of Claim~\ref{nexttolastclaimlangtwoplus}, the subcases where $\set{c,d}\notin M$ do not occur. 
\item We replace the proof of Claim~\ref{lastclaimlangtwoplus} with the following argument. Recall that $\phi$ is $(\all{c}{d})$. We assume that $\semantics{c}\subseteq\semantics{d}$,
and we need to show that $\Gamma\provesA \all{c}{d}$.  By the Truth Lemma, $\set{c}\in\semantics{c}$.
Thus,  $\set{c}\in\semantics{d}$.  By the Truth Lemma again, $\Gamma\provesA \all{c}{d}$. \qedhere
\end{enumerate}
\end{proof}

 \section{$\langthree$ and $\langthreesecond$: Adding the term former $(\someterm{r}{x})$ to $\langone$ and $\langtwo$}
 \label{section-three}
 
 In this section, we study $\langthree$, the language with term formers 
 $(\allterm{r}{x})$ and $(\someterm{r}{x})$, 
 and with sentence former $(\all{x}{y})$. 
 We also study 
 the larger language $\langthreesecond$ which adds the sentence former 
  $(\some{x}{y})$. 
  
The language $\langthreesecond$ has essentially already been studied by 
McAllester and Givan in~\cite{logic:mcA+G92}, but that paper is primarily concerned with complexity results rather than completeness results.   What McAllester and Givan would call a \emph{quantifier-free atomic formula without constants}
is exactly what we call a sentence of $\langthreesecond$.   What they would call
a \emph{quantifier-free literal  without constants}
is either an   $\langthreesecond$ sentence $\phi$ or its negation $\nott\phi$.
They show that the satisfiability problem for sets $\Gamma$ of literals which determine existentials is in $\PTIME$.
And from this, they derive that the 
 satisfiability problem for sets $\Gamma$ of literals (which perhaps do not determine existentials) of literals is in $\NPTIME$. 
 Thus, their result implies that the  consequence  relation for $\langthreesecond$ is in $\CONP$.
 They prove a matching hardness result as well, and so the 
 consequence  relation for $\langthreesecond$ is $\CONP$ complete. 

We show that the consequence relation  for $\langthree$
is $\CONP$ hard.  
  Our proof  is
based on the one in~\cite{logic:mcA+G92}, and
the result here is  a slight improvement on~\cite{logic:mcA+G92} because $\langthree$ is a little weaker than their language.
As a corollary, if {\sc P} $\neq$ {\sc \NP}, then there is no boundedly complete syllogistic proof system for  $\langthree$
 or any language larger than it.
 
In  Sections~\ref{section-completeness-langthree} and~\ref{section-completeness-langthreesecond} we formulate proof systems and
 obtain completeness results
for $\langthree$ and
$\langthreesecond$. We also reprove the $\CONP$ decidability of $\langthreesecond$ by a polynomial-size countermodel construction.

 \subsection{$\CONP$ hardness of the consequence relation of $\langthree$}
 \label{section-one-sentence-former}

\begin{theorem} The  problem of deciding whether $\Gamma\not\models\phi$,
for $\Gamma\cup\set{\phi}$ a finite set of $\langthree$-sentences, is $\NPTIME$ hard.
\label{theorem-McA-G}
\end{theorem}

\begin{proof}
We use a reduction from the \emph{one-in-three positive $3$-SAT} problem
first studied by Schaefer~\cite{Schaefer}.
This problem is defined as follows.    We are given a set $\mathcal{S}$  of 
 clauses of the form $U\lor V\lor W$, where $U$, $V$, and $W$ are distinct.  (Note that negation is not used.)
The problem is to find a truth assignment $f$ to the variables making exactly one variable
in each clause $\true$ and the other two variables $\false$. 
 We call this a \emph{$1$-valued} assignment for ${\cal S}$.
This problem was shown to be $\NPTIME$ complete in Schaefer~\cite{Schaefer}.

We define a set
$\Gamma = \Gamma({\cal S})$ below, in  two steps.   We use
nouns which correspond to the variables of ${\cal S}$, writing $u$ for the noun corresponding with $U$,
etc. $\Gamma$ also uses a number of other nouns and verbs.  It is defined as follows:
\begin{enumerate}[(1)]
\item  For each clause $c\in {\cal S}$, say $c = U\lor V\lor W$,
put the following sentences in $\Gamma$:
\[
\begin{array}{ll}
\all{\startt}{(\allterm{r^1_c}{u})} &
\all{(\someterm{r^1_c}{u})}{y_c} \\
\all{y_c}{(\allterm{r^2_c}{v})} &
\all{(\someterm{r^2_c}{v})}{z_c} \\
\all{z_c}{(\allterm{r^3_c}{w})} &
\all{(\someterm{r^3_c}{w})}{\finish} \\
\end{array}
\]
Here $\startt$ and $\finish$ are new nouns (not varying with the clause),
  $y_c$ and $z_c$ are also new nouns (these do vary with $c$), and
$r^1_2$, $r^2_c$, and $r^3_c$ are new verbs.
\item  Let $P$ and $Q$ be any two distinct variables which occur together in
some clause $c$.   Then add to $\Gamma$ the sentence $\phi_{p,q}$:
\[
\all{(\allterm{r_{p,q}}{p})}{(\someterm{r_{p,q}'}{q})}.
\]
Here $r_{p,q}$ and $r'_{p,q}$ are new verbs.
(By symmetry, we also add $\phi_{q,p}$.)
\end{enumerate}
So if ${\cal S}$ has $k$ clauses, then the first point will add $2 + 2k$ new nouns
and $3k$ new verbs.   The second clause will add  at most
$2\cdot \binom{3k}{2} < 18 k^2$ new verbs.

\begin{claim}
${\cal S}$ has a
$1$-valued assignment iff $\Gamma\not\models \all{\startt}{\finish}$.
\end{claim}

\begin{proof} 
In one direction, assume that $\Model\models \Gamma$
and $\Model\not\models \all{\startt}{\finish}$.
   Define a truth assignment $f$
by $f(U) = \false$   iff $\semantics{u} \neq \emptyset$.
Consider a clause $c = U\lor V\lor W$ of ${\cal S}$.
If $f(U) = f(V) = f(W) = \false$,  then 
$\semantics{u}$, $\semantics{v}$, and $\semantics{w}$ are all non-empty.
By the sentences in (1), 
\[
\semantics{\startt} \subseteq  \semantics{y_c} \subseteq   \semantics{z_c} \subseteq  
\semantics{\finish}. 
\]
 But this contradicts that  $\Model\not\models \all{\startt}{\finish}$.  Thus we know that at least one 
variable in $c$ is assigned the value $\true$ by $f$.   We claim that only one variable can be $\true$.
For suppose towards a contradiction that (for example)  $U \neq V$ but $f(U) = f(V) = \true$.
Then $\semantics{u} = \semantics{v} = \emptyset$.   So $\semantics{\allterm{r_{p,q}}{u}} = M$
and $\semantics{\someterm{r'_{p,q}}{v}} = \emptyset$.   
By the sentence $\phi_{p,q}$  in point (2), $M$ is empty.  But this is impossible, since 
 $\Model\not\models  \all{\startt}{\finish}$.

Conversely, suppose $f$ is $1$-valued on ${\cal S}$.   We must find a model $\Model \models \Gamma$
where  $\Model\not\models \all{\startt}{\finish}$.
Let $M$ be the set of variables $U$ such that $f(U) = \false$, together with $\startt$ and $\finish$.
For a variable $X$, define $\semantics{x} = \emptyset$ if $f(X) = \true$, and
$\semantics{x} = \set{x}$ if $f(X) = \false$.   We also take $\semantics{\startt} = \set{\startt}$,
and $\semantics{\finish} = \set{\finish}$.

We still need to define $\semantics{y_c}$, $\semantics{z_c}$, 
$\semantics{r^1_c}$, $\semantics{r^2_c}$, $\semantics{r^2_c}$ for all clauses $c$,
and also $\semantics{r_{p,q}}$  and $\semantics{r'_{p,q}}$  when $P$ and $Q$ are distinct variables in the same clause.

Suppose that $P$ and $Q$ are distinct variables which happen to belong to 
the same clause.   
We must arrange that $\Model\models\phi_{p,q}$. Set $\semantics{r_{p,q}} = \emptyset$ and $\semantics{r'_{p,q}} = M\times M$. We know that either $f(P) = \false$ or $f(Q) = \false$ (or both).
In the first case, $\semantics{\allterm{r_{p,q}}{p}} = \emptyset$, so $\Model\models\phi_{p,q}$.
In the second case, $\semantics{\someterm{r'_{p,q}}{q}} = M$ so again $\Model\models\phi_{p,q}$.

Finally, we consider the sentences in (1).   There are three cases.

If $f(U) = \true$, $f(V) = \false$, and $f(W) = \false$,
then we have $\semantics{u} = \emptyset$, $\semantics{v} = \set{v}$,
and $\semantics{w} = \set{w}$.
We set 
$\semantics{y_c} = \emptyset$,  $\semantics{z_c} = \emptyset$, 
$\semantics{r^1_c} = M\times M$,  $\semantics{r^2_c} = \emptyset$,  and $\semantics{r^3_c} = \emptyset$.

If $f(U) = \false$, $f(V) = \true$, and $f(W) = \false$,
we have  $\semantics{u} = \set{u}$, $\semantics{v} = \emptyset$,
and $\semantics{w} = \set{w}$.
We set
 $\semantics{y_c} = M$,  $\semantics{z_c} = \emptyset$, 
$\semantics{r^1_c} = M\times M$,  $\semantics{r^2_c} = \emptyset$,  and $\semantics{r^3_c} = \emptyset$.

If $f(U) = \false$, $f(V) = \false$, and $f(W) = \true$, we have 
 $\semantics{u} = \set{u}$, $\semantics{v} = \set{v}$, and $\semantics{w} = \emptyset$.
We set
$\semantics{y_c} = M$,  $\semantics{z_c} = M$, 
$\semantics{r^1_c} = M\times M$,  $\semantics{r^2_c} = M\times M$,  and $\semantics{r^3_c} = \emptyset$.

In all cases, the resulting model $\Model$ satisfies all sentences in (1), hence all sentences in $\Gamma$.
And in all cases,  $\Model\not\models \all{\startt}{\finish}$.
\end{proof}

The claim concludes the proof Theorem~\ref{theorem-McA-G}.
\end{proof}

\subsection{Completeness and $\CONP$ decidability for $\langthreesecond$}
\label{section-completeness-langthreesecond}

We first present a sound and complete proof system for $\langthreesecond$, because it is actually a bit simpler than $\langthree$, and mirrors more closely our work from Section~\ref{sec-langtwo}.  The rules are in 
 Figure~\ref{fig-some-rules}. We return to $\langthree$ in Section~\ref{section-completeness-langthree} below. 
 
Since the consequence relations for  $\langthree$ and $\langthreesecond$ are $\CONP$ hard, by Theorem~\ref{theorem-ptime} we cannot hope for a
boundedly complete syllogistic proof system for these languages (unless {\sc P}$=${\sc NP}). 
We regard it as unlikely that they admit any sound and complete complete syllogistic proof system.
Instead, we settle for a proof system with (\casesrule) from Section~\ref{sec-cases}, as well as a variant, (\casesruleone). Figure~\ref{fig-some-rules} gives proof rules for this logic.

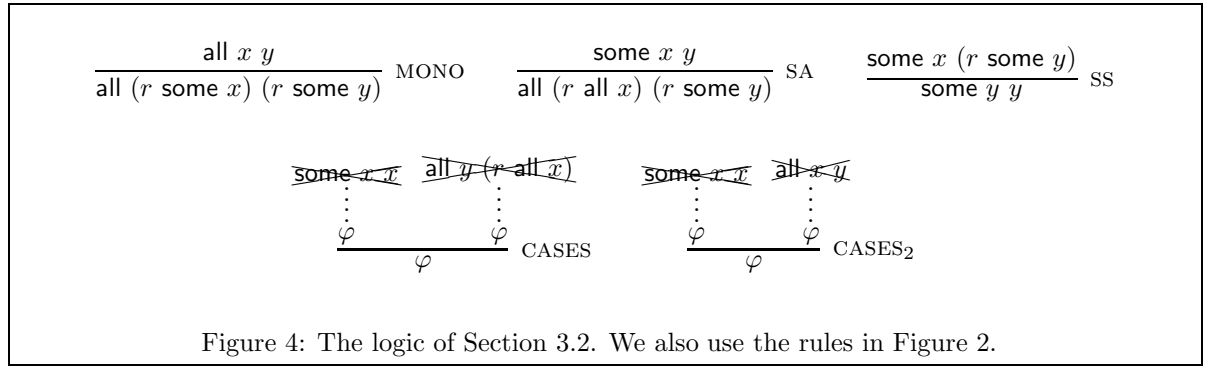
\begin{figure}[h]
\begin{mathframe}
\begin{gather*}
\infer[\rone]{\all{(\someterm{r}{x})}{(\someterm{r}{y})}}{\all{x}{y}} \qquad
\infer[\rtwo]{\all{(\allterm{r}{x})}{(\someterm{r}{y}})}{\some{x}{y}} \qquad
\infer[\rthree]{\some{y}{y}}{\some{x}{(\someterm{r}{y})}}\\\\
\infer[\casesrule]{\phi}{
\infer*{\phi}{\xcancel{\some{x}{x}}} & \infer*{\phi}{\xcancel{\all{y}{(\allterm{r}{x})}}}} \qquad
\infer[\casesruleone]{\phi}{\infer*{\phi}{\xcancel{\some{x}{x}}} & \infer*{\phi}{\xcancel{\all{x}{y}}}}
\end{gather*}
\caption{The logic of Section~\ref{section-completeness-langthreesecond}.
We also use the rules in Figure~\ref{baserules}.
   \label{fig-some-rules}}
\end{mathframe}
\end{figure}

In this section, we write $\proves$ for provability in the system with rules (\axiom), (\barbara), (\anti), (\someone), (\sometwo), (\darii), (\rone), (\rtwo), (\rthree), (\casesrule), and (\casesruleone). Given a theory $\Gamma$, we write $x\leq y$ when $\Gamma\proves \all{x}{y}$. We say that $\Gamma$ \emph{determines existentials} for a set of terms $\Terms$ if for every $x\in \Terms$, either $\Gamma\proves \some{x}{x}$, or else $\Gamma\proves \all{x}{y}$ and $\Gamma\proves\all{y}{(\allterm{r}{x})}$ for all terms $y\in \Terms$ and all verbs $r$.

\paragraph{The canonical model}
Let $\Gamma$ be a theory, and let $\Terms$ be a set of terms, closed under subterms as usual. Define 
\[
M = \set{\pair{x,y,Q}\in \Terms\times \Terms \times \set{\forall,\exists}: \Gamma\proves\some{x}{y}}.
\]
We define a model $\Model(\Gamma,\Terms)$ with domain $M$ by setting 
\begin{align*}
\pair{x,y,Q}\in \semantics{p} & \quadiff
 x\leq p \mbox{ or } y\leq p \\ 
    \pair{x_1, y_1, Q_1}
  \semantics{r} \pair{x_2,y_2, Q_2} & \quadiff \mbox{for some $z_1\in \set{x_1,y_1}$ and $z_2\in \set{x_2,y_2}$, either}\\
  &\qquad\quad \text{(a) }z_1\leq (\allterm{r}{z_2}), \text{or}\\
 & \qquad\quad \text{(b) }
 Q_2 = \exists, x_2 = y_2, \text{and } z_1\leq (\someterm{r}{z_2}).
\end{align*}

\begin{lemma}[Truth Lemma]
If $\Gamma$ determines existentials for $\Terms$, then in $\Model'(\Gamma,\Terms)$, for all $t\in \Terms$,  
\begin{equation}
\label{eq-lemma-determines-existentials-Truth-Lemma-redux2}
\semantics{t}  =    \set{\pair{x,y,Q}\in M: x\leq p\text{ or }y\leq p}.
\end{equation}
\label{lemma-determines-existentials-Truth-Lemma-redux}
\end{lemma}
\begin{proof}   
The proof is by induction on $t$.
For a noun $p$, this is by definition of the model.
 
\paragraph{The induction step for $(\allterm{r}{t})$}
Since $\Terms$ is closed under subterms, $t\in \Terms$. Take some $\pair{c_1,c_2,Q}$ such that $c_i\leq (\allterm{r}{t})$ for some $i\in\set{1,2}$.
We show that $\pair{c_1,c_2,Q}\in\semantics{\mbox{\sf $r$ all $t$}}$.
For this, suppose 
$\pair{d_1,d_2,Q'}\in\semantics{t}$.
By induction hypothesis, $d_j\leq t$ for some $j\in\set{1,2}$.
Using (\anti), 
 $(\allterm{r}{t})\leq (\allterm{r}{d_j})$, so $c_i\leq (\allterm{r}{d_j})$. Thus,
$\pair{c_1,c_2,Q}\semantics{r}\pair{d_1,d_2,Q'}$.
 Since $\pair{d_1,d_2,Q'}$ was arbitrary, we have shown that 
 $\pair{c_1,c_2,Q}\in \semantics{\allterm{r}{t}}$, as desired.
 
 In the other direction, suppose that  $\pair{c_1,c_2,Q}\in \semantics{\mbox{\sf $r$ all $t$}}$. If $\Gamma\proves \some{t}{t}$, then $\pair{t,t,\forall}\in M$. By induction, since $t\leq t$, $\pair{t,t,\forall}\in \semantics{t}$, so  $\pair{c_1,c_2,Q}\semantics{r}\pair{t,t,\forall}$.  Since the last component in $\pair{t,t,\forall}$ is $\forall$ rather than $\exists$, case (a) in the definition of $\semantics{r}$ holds, and $c_i\leq (\allterm{r}{t})$ for some $i\in \set{1,2}$.
 
 It remains to consider the case when $\Gamma\not\proves\some{t}{t}$. But since $\Gamma$ determines existentials, $\Gamma\proves \all{y}{(\allterm{r}{t})}$ for all terms $y$, and in particular $c_1\leq (\allterm{r}{t})$.

\paragraph{The induction step for $(\someterm{r}{t})$}
Again, since $\Terms$ is closed under subterms, $t\in \Terms$. Take some $\pair{c_1,c_2,Q}$ such that $c_i\leq (\someterm{r}{t})$. The fact that $\pair{c_1,c_2,Q}\in M$ implies that $\Gamma\proves\some{c_1}{c_2}$. By (\sometwo) and (\someone), $\Gamma\proves\some{c_i}{c_i}$, and by (\darii), $\Gamma\proves \some{c_i}{(\someterm{r}{t})}$. But then 
by ({\rthree}),  $\Gamma\proves\some{t}{t}$, and hence $\pair{t,t,\exists}\in M$. By case (b) in the definition of $\semantics{r}$, $\pair{c_1,c_2,Q}\semantics{r}\pair{t,t,\exists}$. By induction $\pair{t,t,\exists}\in \semantics{t}$, so $\pair{c_1,c_2,Q}\in \semantics{\someterm{r}{t}}$. 

In the other direction, suppose that $\pair{c_1,c_2,Q} \in \semantics{\someterm{r}{t}}$.  Then we have
$\pair{c_1,c_2,Q} \semantics{r}\pair{d_1,d_2,Q'}$ for some
$\pair{d_1,d_2,Q'}\in\semantics{t}$. Since $\pair{d_1,d_2,Q'}\in M$,
$\Gamma\proves\some{d_1}{d_2}$. And also $d_k\leq t$ for some $k\in \set{1,2}$, by induction. 

We first consider case (a) in
the definition of $\semantics{r}$: there are $i$ and $j$ so
that $c_i\leq (\allterm{r}{d_j})$. From $\Gamma\proves\some{d_1}{d_2}$ and $d_k\leq t$, using (\darii), (\someone), and (\sometwo),  $\Gamma\proves \some{d_j}{t}$. By ({\rtwo}), $(\allterm{r}{d_j})\leq (\someterm{r}{t})$, so by (\barbara), $c_i\leq (\someterm{r}{t})$. 

In case (b),   
 $Q' = \exists$,  $d_1 = d_2$,
 and for some $i$, $c_i\leq (\someterm{r}{d_1})$. By (\rone), $(\someterm{r}{d_1})\leq (\someterm{r}{t})$, so $c_i\leq (\someterm{r}{t})$.
\end{proof}

 \begin{lemma}
 Suppose that $\Terms$ contains all subterms of sentences in $\Gamma$ and $\Gamma^*\supseteq \Gamma$ is a theory which determines existentials for $\Terms$. Then $\Model(\Gamma^*,\Terms)\models \Gamma$. \label{lemma-Gamma-Delta-redux}
\end{lemma}

\begin{proof}    
For a sentence $(\all{x}{y})\in \Gamma$, we have $\Gamma^*\proves \all{x}{y}$. 
We are going to apply Lemma~\ref{lemma-determines-existentials-Truth-Lemma-redux} to $\Model(\Gamma^*,\Terms)$ (that is, the $\leq$ symbol here is for provability in $\Gamma^*$).
  If $\pair{c_1,c_2,Q}\in \semantics{x}$  in $\Model(\Gamma^*,\Terms)$, then $c_i\leq x$ for some $i\in \set{1,2}$. But then also $c_i\leq y$, so $\pair{c_1,c_2,Q}\in \semantics{y}$.

For a sentence $(\some{x}{y})\in \Gamma$, we have $\Gamma^*\proves \some{x}{y}$, so $\pair{x,y,\forall}\in M$. And $\pair{x,y,\forall}\in \semantics{x}\cap \semantics{y}$ in $\Model(\Gamma^*,\Terms)$.
\end{proof}

\begin{theorem}
$\Gamma\models\phi$ iff $\Gamma\proves\phi$.
\label{theorem-completeness-langthreesecond}
\end{theorem}

\begin{proof}  The soundness is easy, with soundness of (\casesruleone) following just as in the proof of Lemma~\ref{lemma-soundness-cases}. 

The argument for completeness is the same as we saw in the proof of Theorem~\ref{theorem-completeness-cases}.
Let $\Terms$ be the set of all subterms of sentences in $\Gamma\cup \set{\phi}$. We assume that $\Gamma\not\proves\phi$ and show that
  $\Gamma\not\models\phi$.
We may  pass to a maximal extension $\Gamma^*\supseteq \Gamma$ with the property that $\Gamma^*\not\proves\phi$.
 It follows from $(\casesrule)$ and $(\casesruleone)$ that $\Gamma^*$ determines existentials for $\Terms$, just as in the proof of Theorem~\ref{theorem-completeness-cases}.    By Lemma~\ref{lemma-Gamma-Delta-redux},
  $\Model(\Gamma^*,\Terms)\models\Gamma$.   We claim that $\phi$ is false in this model.
  
Case 1: $\phi$ is $(\all{x}{y})$. Since $\Gamma^*$ determines existentials for $\Terms$ and $\Gamma^*\not\proves \all{x}{y}$, we have $\Gamma^*\proves\some{x}{x}$. So $\pair{x,x,\forall}\in M$. By Lemma~\ref{lemma-Gamma-Delta-redux}, $\pair{x,x,\forall}\in \semantics{x}\setminus \semantics{y}$, since $x\leq x$ but $x\not\leq y$. 

Case 2: $\phi$ is $(\some{x}{y})$. Suppose towards a contradiction 
that $\Model(\Gamma^*,\Terms)\models\phi$.  Specifically, let
 $\pair{d_1,d_2,Q}\in \semantics{x} \cap \semantics{y}$.
Then $\Gamma^*\proves\some{d_1}{d_2}$, 
and also there are $i$ and $j$ such that 
 $d_i\leq x$ and $d_j\leq y$. Using (\someone), (\sometwo), and (\darii), $\Gamma^*\proves\some{x}{y}$.
\end{proof}

The proof shows that if $\Gamma\not\proves \phi$, then there is a countermodel of size $O(n^2)$, where $n$ is the complexity of $\Gamma\cup \set{\phi}$. Since $\Model\models \Gamma$ and $\Model \not\models \phi$ can be checked in time polynomial in the size of $\Model$ and the complexity of $\Gamma\cup \set{\phi}$, this shows that the consequence relation for $\langthreesecond$ is in $\CONP$.

\subsection{Completeness for $\langthree$}
\label{section-completeness-langthree}

In $\langthree$, we do not have the sentence former $(\some{x}{y})$, which was used in the previous section to formulate the condition that $\Gamma$ determines existentials. Nevertheless, we are able to follow the same strategy as for $\langthreesecond$ to prove a completeness theorem. The key observation is that for a term $x$ and a nonempty model $\Model$, $\semantics{x}\neq \emptyset$ in $\Model$ if and only if $\Model \models \all{(\allterm{r}{x})}{(\someterm{r}{x})}$ for every verb $r$. We use the collection of all sentences of this form as a replacement for $(\some{x}{x})$.

Figure~\ref{fig-3-rules} gives our proof system for this logic. 
The soundness of (\rone) is immediate. For (\mixrule), note that if $\Nodel\models \all{(\allterm{r}{y})}{(\someterm{r}{y})}$, then either $N = \emptyset$, or $\semantics{y} \neq \emptyset$.
If $N=\emptyset$, the conclusion of the rule holds.   And if $\semantics{y} \neq \emptyset$
and $\Nodel\models\all{y}{(\someterm{r}{x})}$, then $\semantics{x} \neq \emptyset$ also.
Again, the conclusion of the rule follows. 
The soundness of the (\casesrule) variants follow as in 
Lemma~\ref{lemma-soundness-cases}, using the above observation about sentences of the form $\all{(\allterm{r}{y})}{(\someterm{r}{y})}$.

Note that in (\mixrule), the verb $s$ appearing in the conclusion may be different than the verb $r$ appearing in the premises. And in (\casesrulethree), the withdrawn premises are $(\all{(\allterm{r}{x})}{(\someterm{r}{x})})$ and $(\all{y}{(\allterm{s}{x})})$, where the verbs $r$ and $s$ may be different. 

As usual, we write $\proves$ for provability in this system, and given a theory $\Gamma$, we write $x\leq y$ when $\Gamma\proves \all{x}{y}$.

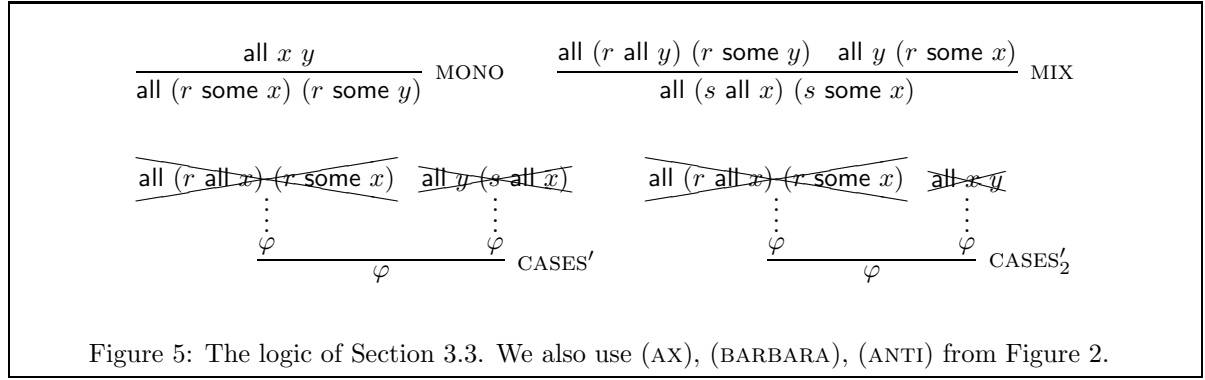
\begin{figure}[h]
\begin{mathframe}
\begin{gather*}
\infer[\rone]{\all{(\someterm{r}{x})}{(\someterm{r}{y})}}{\all{x}{y}} \qquad
\infer[\mixrule]{\all{(\allterm{s}{x})}{(\someterm{s}{x})}}{\all{(\allterm{r}{y})}{(\someterm{r}{y})} & \all{y}{(\someterm{r}{x})}}
\\
\\
\infer[\casesrulethree]{\phi}{
\infer*{\phi}{\xcancel{\all{(\allterm{r}{x})}{(\someterm{r}{x})}}}
&
\infer*{\phi}{\xcancel{\all{y}{(\allterm{s}{x})}}}
}
 \qquad
\infer[\casesruletwo]{\phi}{
\infer*{\phi}{\xcancel{\all{(\allterm{r}{x})}{(\someterm{r}{x})}}}
&
\infer*{\phi}{\xcancel{\all{x}{y}}}
}
\end{gather*}
\caption{The logic of Section~\ref{section-completeness-langthree}.
 We also use   ({\axiom}), (\barbara), (\anti)  from Figure~\ref{baserules}.
   \label{fig-3-rules}}
\end{mathframe}
\end{figure}

\begin{definition} 
Let $\Gamma$ be a theory and $\Terms$ a set of terms. A term $x\in \Terms$ is \emph{effectively non-empty (for $\Gamma$)} if for all verbs $r$,
$\Gamma\proves \all{(\allterm{r}{x})}{(\someterm{r}{x})}$.
And $x$ is \emph{effectively empty (for $\Gamma$)} if for all $y\in\Terms$,  $\Gamma\proves \all{x}{y}$ and $\Gamma\proves \all{y}{(\allterm{r}{x})}$ for all  verbs $r$.

$\Gamma$ \emph{effectively determines existentials for $\Terms$} if
every $x\in\Terms$ is either effectively empty or effectively non-empty for $\Gamma$.
\label{def-determines-existentials-3}
\end{definition}
 
 \paragraph{The canonical model}
Let $\Gamma$ be a theory, and let $\Terms$ be a set of terms. Define 
\[
M = \set{\pair{x,Q} \in \Terms\times \set{\forall,\exists} :  
\mbox{$x$ is effectively non-empty for $\Gamma$}}
\]
We define a model $\Model(\Gamma,\Terms)$ with domain $M$ by setting 
\begin{align*}
\pair{x,Q}\in \semantics{p} & \quadiff
 x\leq p \\ 
    \pair{x, Q}
  \semantics{r} \pair{y,Q'} & \quadiff \mbox{either (a) $x\leq (\allterm{r}{y})$,} \\
& \qquad \quad \mbox{or (b) $Q' = \exists$, and $x\leq {(\someterm{r}{y})}$.}
\end{align*}

\begin{lemma}[Truth Lemma] \label{yet-another-truth-lemma} Assume that $\Gamma$ effectively determines existentials for $\Terms$.
Then for all $x\in\Terms$, 
\[
\semantics{x} = \set{\pair{y,Q}\in M :y \leq x}.
\]
\end{lemma}
\begin{proof} 
By induction on $x$.  When $x$ is a noun, this follows immediately from the definition. 

Here is the induction step for $(\allterm{r}{x})$. 
Suppose that $\pair{y,Q}\in \semantics{\allterm{r}{x}}$.  
If $x$ is effectively empty, then 
 ${y} \leq (\allterm{r}{x})$, and so we are done.    
If $x$ is effectively non-empty, then $\pair{x,\forall}\in M$. Since $x\leq x$, by induction $\pair{x,\forall}\in \semantics{X}$, and $\pair{y,Q}\semantics{r}\pair{x,\forall}$. Then case (a) holds and again $y\leq (\allterm{r}{x})$. 

In the other direction, suppose that $y\leq ( \allterm{r}{x})$.
We show that for any $\pair{y,Q}\in M$, $\pair{y,Q}\semantics{r}\pair{z,Q'}$ for all
$\pair{z,Q'}\in\semantics{x}$.  By induction, $z \leq x$.  
Using (\anti), $ ( \allterm{r}{x}) \leq  (\allterm{r}{z})$.
 Thus, $y\leq  ( \allterm{r}{z})$, and so indeed $\pair{y,Q}\semantics{r}\pair{z,Q'}$.

Finally, we have the induction step for $(\someterm{r}{x})$.  
Suppose that $\pair{y,Q}\in \semantics{\someterm{r}{x}}$.  Then there is some $\pair{z,Q'}\in\semantics{x}$ such that 
$\pair{y,Q}\semantics{r}\pair{z,Q'}$.   So $z$ is effectively non-empty, and by induction $z\leq x$. 

Case 1: $Q'=\exists$. Then $y\leq (\someterm{r}{z})$. By (\rone), $(\someterm{r}{z})\leq (\someterm{r}{x})$, so also $y\leq (\someterm{r}{x})$, as was to be shown. 

Case 2: $Q'=\forall$. Then $y\leq (\allterm{r}{x})$. By (\anti), $(\allterm{r}{x})\leq (\allterm{r}{z})$, so $y\leq (\allterm{r}{z})$. The fact that $z$ is effectively non-empty means that $(\allterm{r}{z})\leq (\someterm{r}{z})$. And by (\rone), as observed above, $(\someterm{r}{z})\leq (\someterm{r}{x})$. So again we have $y\leq (\someterm{r}{x})$.

In the other direction, let 
$y\in\Terms$ be such that $y\leq (\someterm{r}{x})$. Suppose that $\pair{y,Q}\in M$.
Then $y$ is effectively non-empty. In particular, $(\allterm{r}{y})\leq (\someterm{r}{y})$, and by (\mixrule), for every verb $s$, $(\allterm{s}{x})\leq (\someterm{s}{x})$, so $x$ is effectively non-empty. Then $\pair{x,\exists}\in M$. By case (b), $\pair{y,Q}\semantics{r}\pair{x,\exists}$, and by induction $\pair{x,\exists}\in \semantics{x}$, so $\pair{y,Q}\in \semantics{\someterm{r}{x}}$.
\end{proof}

\begin{lemma}
Suppose $\Terms$ contains all subterms of sentences in $\Gamma$ and $\Gamma^*\supseteq \Gamma$ is a theory which effectively determines existentials for $\Terms$. Then $\Model(\Gamma^*,\Terms)\models \Gamma$. 
\label{langthreemodel}
\end{lemma}
\begin{proof}
For a sentence $(\all{x}{y})\in \Gamma$, we have $\Gamma^*\proves \all{x}{y}$, so $x\leq y$. If $\pair{z,Q}\in \semantics{x}$ in $\Model(\Gamma^*,\Terms)$, then $z\leq x$. But then also $z\leq y$, so $\pair{z,Q}\in \semantics{y}$. 
\end{proof}

\begin{theorem}
$\Gamma\models\phi$ iff $\Gamma\proves\phi$.
\label{theorem-completeness-langthree}
\end{theorem}
\begin{proof}
We discussed the soundness when we introduced the rules.

 For completeness, let $\Terms$ be the set of all subterms of sentences in $\Gamma\cup \set{\phi}$. We assume that $\Gamma\not\proves\phi$ and show that
  $\Gamma\not\models\phi$. 
We may  pass to a maximal extension $\Gamma^*\supseteq \Gamma$ with the property that $\Gamma^*\not\proves\phi$.
 It follows from $(\casesrulethree)$ and $(\casesruletwo)$ that $\Gamma^*$ effectively determines existentials for $\Terms$, just as in the proof of Theorem~\ref{theorem-completeness-cases}.    By Lemma~\ref{langthreemodel},
  $\Model(\Gamma^*,\Terms)\models\Gamma$.  
  
   We claim that $\phi$ is false in this model.
Write  $\phi$ as $(\all{x}{y})$. Since $\Gamma^*$ effectively determines existentials for $\Terms$ and $\Gamma^*\not\proves \all{x}{y}$, $x$ is effectively nonempty for $\Gamma^*$. So $\pair{x,\forall}\in M$. By Lemma~\ref{yet-another-truth-lemma}, $\pair{x,\forall}\in \semantics{x}\setminus \semantics{y}$, since $x\leq x$ but $x\not\leq y$. 
 \end{proof}
 
 The proof shows that if $\Gamma\not\proves \phi$, then there is a countermodel of size $O(n)$, where $n$ is the complexity of $\Gamma\cup \set{\phi}$.

\section{$\langfour$ and $\langfoursecond$: Adding term complementation to $\langone$ and $\langtwo$}
\label{section-term-negation}

This section extends $\langone$ and $\langtwo$ by adding term complements. 
That is, we extend the syntax of $\lang$ so that whenever $t$ is a term, $\notterm{t}$ is a term, and we extend the semantics so that in a model $\Model$,  $\semantics{\notterm{t}} = \Model\setminus \semantics{t}$. If we add term complements to $\langone$, we get $\langfour$, and if we add term complements to $\langtwo$, we get $\langfoursecond$.

\subsection{$\CONP$ hardness of the consequence relation of $\langfour$}\label{3sat}

We begin with a negative complexity-theoretic result, reducing $3$-SAT to the relation $\Gamma\not\models\varphi$ in $\langfour$.

Let $BV = \set{P_i : i \in N}$ be a set of boolean variables.
Suppose we have an instance of $3$-SAT, $c_1\land \dots\land c_k$, where each clause $c_i$ has the form $U_i \lor V_i\lor W_i$, where $U_i$, $V_i$ and $W_i$ are literals: variables in $BV$ or their negations. 

Then we consider the language with nouns 
\[
\set{p : P\in BV} \cup \set{q} \cup \set{y_i,z_i: 1\leq i\leq k}
\]
 and verbs $\set{r_i: 1\leq i\leq k}$. 
Notice that we write $p$ for the noun corresponding to the variable $P$. We will also write $u$ for the literal $U$, where if $U$ is a variable $P$, then $u$ is the noun $p$, and if $U$ is a negated variable $\lnot P$, then $u$ is the term $\notterm{p}$. 

For each clause $c_i = U_i\lor V_i\lor W_i$, we define the following sentences:
\begin{align*}
\psi^i_1 &= \all{\notterm{u_i}}{(\allterm{r_i}{y_i})} \\
\psi^i_2 &= \all{\notterm{v_i}}{(\allterm{r_i}{\notterm{y_i}})}\\
\psi^i_3 &= \all{(\allterm{r_i}{z_i})}{w_i}
\end{align*}
And we define 
\begin{align*}
\Gamma &= \set{\psi^i_1,\psi^i_2,\psi^i_3: 1\leq i\leq k}\\
\varphi &= \all{q}{\notterm{q}}
\end{align*}

\begin{lemma}\label{nonempty}
$\Gamma\not\models \varphi$ if and only if $\Gamma$ has a nonempty model.
\end{lemma}
\begin{proof}
Suppose $\Gamma$ has a model $\Model$ with nonempty domain $M$. 
Let $\Model'$ be a model with domain $M$ and the same interpretations of all of the nouns and verbs, 
except for $q$, which we interpret as all of $M$. Since the sentences in $\Gamma$ do not mention $q$, 
we still have $\Model' \models \Gamma$. And for any $x\in M$, we have $x\in \semantics{q} = M$ and 
$x\notin \semantics{\notterm{q}} = \emptyset$, so $\Gamma\not\models \varphi$. Conversely, suppose every model of $\Gamma$ is empty. 
Then in any model $\Model$ of $\Gamma$, we have $\semantics{q} = \emptyset \subseteq \semantics{\notterm{q}} = \emptyset$, so $\Gamma\models \varphi$.
\end{proof}

\begin{theorem}
$\Gamma\not\models \varphi$ if and only if $c_1\land \dots \land c_k$ is satisfiable.
\end{theorem}
\begin{proof}
Suppose $\Gamma\not\models \varphi$. By Lemma~\ref{nonempty}, $\Gamma$ has a model $\Model$ with nonempty domain $M$. Let $x\in M$. Define a truth assignment $f$ for the proposition letters, by 
\[ 
f(P) = \begin{cases} \true &\text{if }x\in \semantics{p}\\ 
\false &\text{if }x\notin \semantics{p}.\end{cases}
\]
The truth assignment extends to literals in the natural way: $f(\lnot P) = \true$ if $f(P) = \false$, and $f(\lnot P) = \false$ if $f(P) = \true$. Note that if $U$ is a literal with corresponding term $u$, then we have $x\in \semantics{u}$ if and only if $f(U) = \true$. 

We check that each clause $c_i = U_i\lor V_i \lor W_i$  is satisfied. If $f(U_i) = \true$ or $f(V_i) = \true$, then $c_i$ is satisfied. So suppose that $f(U_i) = f(V_i) = \false$. Then, since $\semantics{z_i}\subseteq \semantics{y_i}\cup \semantics{\notterm{y_i}}$,
\[
x\in \semantics{\notterm{u_i}} \cap \semantics{\notterm{v_i}} \subseteq 
\semantics{\allterm{r_i}{y_i}} \cap \semantics{\allterm{r_i}{\notterm{y_i}}}\subseteq 
\semantics{\allterm{r_i}{z_i}} \subseteq  \semantics{w_i}.
\]
So $x\in \semantics{w_i}$, and $f(w_i) = \true$, and $c_i$ is satisfied. 

Conversely, suppose $f$ is a truth assignment such that $c_1\land\dots \land c_k$ is satisfied. By Lemma~\ref{nonempty}, it suffices to build a nonempty model of $\Gamma$. Let 
$\Model = \set{x}$, and for each proposition letter $P$, set 
\[
\semantics{p} = \begin{cases} \set{x} & \text{if }f(P) = \true\\ 
\emptyset & \text{if }f(P) = \false
\end{cases}
\]

Note that again, if $U$ is a literal with corresponding term $u$, we have $x\in \semantics{u}$ if and only if $f(U) = \true$. 

Consider a clause $c_i = U_i\lor V_i\lor W_i$. We must define the interpretations of $y_i$, $z_i$, and $r_i$, so that $\psi^i_1$, $\psi^i_2$, and $\psi^i_3$ are satisfied for $1\leq i \leq k$. 

Case 1: $x\in \semantics{w_i}$. Set $\semantics{r_i} = \set{(x,x)}$. The interpretations of $y_i$ and $z_i$ are irrelevant. Indeed, $\psi^i_1$, $\psi^i_2$, and $\psi^i_3$ are satisfied, since $\semantics{\allterm{r_i}{y_i}} = \semantics{\allterm{r_i}{\notterm{y_i}}} = \semantics{w_i} = M$.

Case 2: $x\notin \semantics{w_i}$. Set $\semantics{z_i} = \set{x}$ and $\semantics{r_i} = \emptyset$. Then $\psi^i_3$ is satisfied, since $\semantics{\allterm{r_i}{z_i}} = \semantics{w_i} = \emptyset$. Since the clause $c_i = U_i \lor V_i\lor W_i$ is satisfied, $x$ must be in at least one of $\semantics{u_i}$ or $\semantics{v_i}$. If it's in both, we're done (and the interpretation of $y_i$ is irrelevant), since $\semantics{\notterm{u_i}} = \semantics{\notterm{v_i}} = \emptyset$. Otherwise, set 
\[\semantics{y_i} = \begin{cases} \emptyset & \text{if }x\notin \semantics{u_i}, x\in \semantics{v_i}\\ \set{x} & \text{if }x\in \semantics{u_i}, x\notin \semantics{v_i}.\end{cases}\]
In the first case, $\Model\models \psi^i_1$, since $\semantics{\allterm{r_i}{y_i}} = M$, and $\Model \models \psi^i_2$, since $\semantics{\notterm{v_i}} = \emptyset$. In the second case, $\Model \models \psi^i_2$, since $\semantics{\notterm{u_i}} = \emptyset$, and $\Model \models \psi^i_2$, since $\semantics{\allterm{r_i}{\notterm{y_i}}} = M$.
\end{proof}

  \subsection{Completeness for the extended languages $\langfour^+$ and $\langfoursecond^+$}
 \label{section-sequent-calclulus}

 We enlarge our syntax of $\langfoursecond$ from sentences $(\all{x}{y})$ and $(\some{x}{y})$ to 
expressions of  the form $[x_1, \dots, x_n]$ and $\pair{x_1,\dots,x_n}$ for $n\geq 1$. Note that this is a departure from the languages we have studied previously, since here we have infinitely many sentence formers, of arbitrary finite length.  We emphasize that $[x_1,\ldots, x_n]$ and $\pair{x_1,\dots,x_n}$ are sentences, not terms. The terms of $\langfoursecond^+$ are the same as the terms of $\langfoursecond$. 

 The  semantics of this new language $\langfoursecond^+$ is: 
 \begin{align*}
 \Model\models [x_1, \ldots, x_n] &\quadiff 
\bigcap_{i=1}^n \semantics{x_i} = \emptyset\\
\Model\models\pair{x_1,\ldots, x_n}
& \quadiff
\bigcap_{i=1}^n \semantics{x_i} \neq \emptyset.
\end{align*}

It is clear that $\pair{x_1,\dots,x_n}$ generalizes $(\some{x}{y})$, since the latter sentence can be translated into $\langfoursecond^+$ as $\pair{x,y}$. To see that $[x_1,\dots,x_n]$ generalizes $(\all{x}{y})$, note that $\bigcap_{i=1}^n\semantics{x_i} = \emptyset$ if and only if for some $j$ (equivalently, for all $j$), $\bigcap_{i\neq j}\semantics{x_i}\subseteq \semantics{\notterm{x_j}}$. So the sentence $(\all{x}{y})$ can be translated into $\langfoursecond^+$ as $[x,\notterm{y}]$.

 Our  proof rules are shown in Figure~\ref{fig-sequents}.
  (\axiom)  is a version of the axiom rule as we have seen it throughout the paper.
   (\res) is named for \emph{resolution}.  (But please note that the sentence $[x_1, \ldots, x_n]$ is not interpreted disjunctively as in resolution.) 
 The name $(\rel)$ stands for \emph{relational}, since it is the only rule of the system that mentions relations.
 We name (\structural) after \emph{structural rules} of sequent calculi. 
  The side condition on this rule is that each $x_i$ appears in the list $y_1,\ldots, y_n$.
  It implies the usual rules of weakening, contraction, and exchange.
   (\efq) and (\raa)  are our formulations of \emph{ex falso quodlibet} and \emph{reductio ad absurdum}.  In this system, $[y_1,\ldots, y_n]$ and $\pair{y_1,\ldots, y_n}$ are contradictories.
   Given two derivations with contradictory conclusions, one may use (\efq) to put these two together and conclude any sentence of the form $[x_1,\dots,x_n]$. Alternatively, one may use (\raa) to withdraw
   all occurrences of any one assumption of the form $[x_1,\dots,x_n]$, and conclude the contradictory of that assumption, $\pair{x_1,\dots,x_n}$. Note the asymmetry between  $[x_1,\dots,x_n]$ and $\pair{x_1,\dots,x_n}$; this is arranged to as to allow an easy proof-theoretic argument (Corollary~\ref{cor:langfourcompleteness}) that (\axiom), (\res), (\rel), and (\structural) give a sound and complete proof system for the smaller language $\langfour^+$, which only has sentences of the form $[x_1,\dots,x_n]$.

  \begin{figure}[t]
\begin{mathframe}
 \begin{gather*}
 \infer[\axiom]{[x,\notterm{x}]}{}  \qquad
\infer[\res]{[y_1, \ldots, y_n,z_1,\ldots, z_m]}{[x, y_1, \ldots, y_n] & [\notterm{x}, z_1, \ldots,z_m]}
\qquad
 \infer[\rel]{[\allterm{r}{x_1}, \allterm{r}{x_2}, \dots, \allterm{r}{x_{n-1}}, \notterm{\allterm{r}{x_n}}]}{[\notterm{x_1}, \notterm{x_2}, \dots, \notterm{x_{n-1}}, x_n]}\\\\
 \infer[\structural]{[y_1,\ldots, y_n]}{[x_1,\ldots, x_m]} 
 \qquad
 \infer[\efq]{[x_1,\ldots, x_n]}{\pair{y_1,\ldots, y_m}  & [y_1,\ldots,y_m]}
 \qquad
 \infer[\raa]{\pair{x_1,\ldots, x_n}}{\infer*{\pair{y_1,\ldots, y_m}}{\xcancel{[x_1,\ldots, x_n]}}  & \infer*{[y_1,\ldots,y_m]}{\xcancel{[x_1,\ldots, x_n]}} }
 \end{gather*}
 \caption{Rules of the logical system for $\langfoursecond^+$.  The side condition on
 the structural rule (\structural) is that each $x_i$ must be identical to some $y_j$.
 \label{fig-sequents}}
\end{mathframe}
\end{figure}

 \begin{example}
 Here is how the translations of (\barbara) and  (\someone)  are derived in this system:
 \[
 \infer[\res]{[x,\notterm{z}]}
  {
  \infer[\structural]{[\notterm{y}, x]}{[x,\notterm{y}]}  &  [y,\notterm{z}]
  }
  \qquad
  \infer[\raa]{\pair{x,x}}{\pair{x,y} &  \infer[\structural]{[x,y]}{\xcancel{[x,x]}}}
  \]
 And here are   (\darii) and (\anti):
\[
\infer[\raa]{\pair{x,z}}{
\pair{x,y}
&
\infer[\res]{[x,y]}
{
 \infer[\structural]{[z,x]}{\xcancel{[x,z]}}
 &
 \infer[\structural]{[\notterm{z},y]}{[y,\notterm{z}]}
 }
 }
 \qquad
 \qquad
 \infer[\rel]{[\allterm{r}{y}, \notterm{\allterm{r}{x}}]}{
 \infer[\structural]{[\notterm{y},x]}{[x,\notterm{y}]} 
 }
  \]
 
 \end{example}

 \begin{lemma}[Soundness] If $\Gamma\proves \phi$, then $\Gamma\models\phi$.
 \label{lemma-soundness-sequent}
 \end{lemma}
 
 \begin{proof}
The soundness of (\axiom) and (\structural) are clear.  
 For (\res), (\rel), and (\efq), fix a model $\Model$. 
 
 For (\res), assume that $\semantics{x}\cap (\bigcap_{i\leq n} \semantics{y_i}) = \emptyset$ and $\semantics{\notterm{x}}\cap (\bigcap_{j\leq m}\semantics{z_j}) = \emptyset$. Suppose there is some $\ell\in (\bigcap_{i\leq n} \semantics{y_i})\cap  (\bigcap_{j\leq m}\semantics{z_j})$. Then either $\ell\in \semantics{x}$ or $\ell\in \semantics{\notterm{x}}$, which is a contradiction in either case.

 For (\rel), assume that $(\bigcap_{i<n} \semantics{\notterm{x_i}})\cap \semantics{x_n} = \emptyset$, so $\bigcap_{i<n} \semantics{\notterm{x_i}} \subseteq \semantics{\notterm{x_n}}$.
 Suppose there is some $\ell\in (\bigcap_{i<n} \semantics{\allterm{r}{x_i}})\cap \semantics{\notterm{\allterm{r}{x_n}}}$.
 Then there is some $m\in \semantics{x_n}$ such that it is not true that $\ell \semantics{r} m$.
 For all $i<n$, since $\ell\in  \semantics{\allterm{r}{x_i}}$, $m\in\semantics{\notterm{x_i}}$.  But then $m\in \semantics{\notterm{x_n}}$, and this is a contradiction.

 For (\efq), it is vacuously true that if $\Model \models \pair{y_1,\dots,y_m}$ and $\Model \models [y_1,\dots,y_m]$, then $\Model \models [x_1,\dots,x_n]$, since $\pair{y_1,\dots,y_m}$ and $[y_1,\dots,y_m]$ are contradictory. 
 
 The soundness of (\raa) is by induction on the number of instances of (\raa) in the proof tree. Assume that the last use of (\raa) is at the root of the proof tree showing $\Gamma\proves \pair{x_1,\dots,x_n}$. Then we have $\Gamma\cup \set{[x_1,\dots,x_n]}\proves \pair{y_1,\dots,y_m}$ and $\Gamma\cup \set{[x_1,\dots,x_n]}\proves [y_1,\dots,y_m]$. By induction, these deductions are sound, so every model of $\Gamma\cup \set{[x_1,\dots,x_n]}$ satisfies both $\pair{y_1,\dots,y_m}$ and $[y_1,\dots,y_m]$. But these sentences are contradictory, so $\Gamma\cup \set{[x_1,\dots,x_n]}$ has no models. In other words, every model of $\Gamma$ satisfies $\pair{x_1,\dots,x_n}$, as was to be shown.
 \end{proof}

 \begin{definition}  Let $\Gamma$ be a theory, and let $\SS$ be a set of terms.
 \begin{enumerate}
 \item $\Gamma$ is \emph{inconsistent} if it proves both $[x_1,\dots,x_n]$ and $\pair{x_1,\dots,x_n}$ for some list of terms $x_1,\dots,x_n$. Otherwise, $\Gamma$ is \emph{consistent}.
 \item 
  $\SS$ is \emph{$\Gamma$-inconsistent} if there is a list of terms $x_1,\dots,x_n$ from $\SS$
  such that $\Gamma\proves [x_1, \ldots, x_n]$. Otherwise, 
 $\SS$ is \emph{$\Gamma$-consistent}.
 \end{enumerate}
 \end{definition}
 
 \begin{lemma}
 Let $\SS$ be $\Gamma$-consistent.  Then for all $x$, either $\SS\cup\set{x}$ or  $\SS\cup\set{\notterm{x}}$ is  $\Gamma$-consistent.
 \label{lemma-sequent-consistency}
 \end{lemma}
 
\begin{proof}
Suppose not.  
 Then by (\structural), there are $y_1,\ldots, y_n\in \SS$  such that $\Gamma\proves[x,y_1,\ldots, y_n]$,
and there are $z_1,\ldots, z_m\in\SS$  such that $\Gamma\proves[\notterm{x},z_1,\ldots, z_m]$.
By (\res),  $\Gamma\proves[y_1,\ldots, y_n,z_1,\ldots, z_m]$.  This contradicts the $\Gamma$-consistency of $\SS$.
 \end{proof}
 
 \begin{lemma}  Let  $\SS$ be a maximal $\Gamma$-consistent set of terms.  Then 
 for all $x$, exactly one of $x$ or $\notterm{x}$ belongs to $\SS$.
 \label{lemma-exactly-one}
 \end{lemma}
 
 \begin{proof}
 By Lemma~\ref{lemma-sequent-consistency} and maximality, either $x$ or $\notterm{x}$ belongs to $\SS$.
 Both cannot belong to $\SS$, since this would contradict $\Gamma$-consistency, due to (\axiom).
\end{proof}

 \begin{lemma}
  Let $\SS$ be maximal $\Gamma$-consistent, and suppose that $(\allterm{r}{x})\notin \SS$.
 Then there is some maximal $\Gamma$-consistent $\TT$ such that $x\in \TT$;
  and whenever $(\allterm{r}{y})\in \SS$, we have $\notterm{y}\in \TT$.
 \label{lemma-for-sequent-truth-lemma}
 \end{lemma}
 
 \begin{proof}
 Let $\TT_0 = \set{x}\cup \set{\notterm{y} : (\allterm{r}{y})\in \SS}$.
 If $\TT_0$ is $\Gamma$-inconsistent, then by (\structural) there are $y_1, \ldots, y_n$ such that $(\allterm{r}{y_i})\in \SS$ for all $i$ and  $\Gamma\proves [\notterm{y_1}, \ldots, \notterm{y_n},x]$.
  By (\rel), 
 \[\Gamma\proves [\allterm{r}{y_1},\dots,\allterm{r}{y_n}, \notterm{\allterm{r}{x}}].\]
 Since $\SS$ contains each term $(\allterm{r}{y_i})$, by $\Gamma$-consistency it does not contain $(\notterm{\allterm{r}{x}})$. So by 
 Lemma~\ref{lemma-exactly-one}, $\Gamma$ contains $(\allterm{r}{x})$. This is a contradiction.
 
 So $\TT_0$ is $\Gamma$-consistent.  
 By Zorn's Lemma, $\TT_0$ has  a maximal $\Gamma$-consistent extension, say $\TT$.
\end{proof}
 
 \paragraph{The canonical model of $\Gamma$}
 Let $M$  be the set of all maximal $\Gamma$-consistent sets of terms.   We define a model $ \Model(\Gamma)$ with domain $M$ by defining:
 \begin{align*}
 \SS\in \semantics{p} &\quadiff p\in \SS\\
 \SS\semantics{r}\TT &\quadiff \mbox{for some $z\in \TT$, $(\allterm{r}{z}) \in \SS$}.
 \end{align*}

 \begin{lemma}[Truth Lemma] In $\Model(\Gamma)$, for any term $x$, $\semantics{x} = \set{\SS \in M : x \in \SS}$.
 \end{lemma}
 
 \begin{proof}
 By induction on $x$. When $x$ is a noun, this is by the definition of the model. 
 
 Induction step for $(\allterm{r}{x})$: Suppose $(\allterm{r}{x})\in \SS$, and suppose $\TT\in \semantics{x}$. By induction $x\in \TT$, so by the definition of the model, $\SS\semantics{r}\TT$. Thus $\SS\in \semantics{\allterm{r}{x}}$. Conversely, suppose $\SS\in \semantics{\allterm{r}{x}}$, and assume for contradiction that $(\allterm{r}{x})\notin \SS$. By Lemma~\ref{lemma-for-sequent-truth-lemma}, there is some $\TT\in M$ such that $x\in \TT$, and whenever $(\allterm{r}{y})\in \SS$, we have $\notterm{y}\in \TT$, so $y\notin \TT$ by Lemma~\ref{lemma-exactly-one}. By induction, $\TT\in \semantics{x}$, but it is not the case that $\SS\semantics{r}\TT$, which is a contradiction.  
 
 Induction step for $\notterm{x}$: For any $\SS\in M$, we have $\SS\in \semantics{\notterm{x}}$ if and only if $\SS\notin \semantics{x}$. By induction, this is equivalent to $x\notin \SS$. And by Lemma~\ref{lemma-exactly-one}, $x\notin \SS$ if and only if $\notterm{x}\in \SS$. 
 \end{proof}
 
 \begin{lemma} If $\Gamma$ is consistent, then $\Model(\Gamma)\models \Gamma$. 
 \label{lemma-can-works}
 \end{lemma}
 
 \begin{proof}
 Let $\varphi\in \Gamma$. First, suppose $\varphi = [x_1,\ldots, x_n]$. Suppose for contradiction that there exists $\SS\in\bigcap_i\semantics{x_i}$.  Then by the Truth Lemma, $x_i\in \SS$ for all $i$, contradicting $\Gamma$-consistency of $\SS$. 
 
Now suppose $\varphi = \pair{x_1,\dots,x_n}$. We claim that the set $\set{x_1,\ldots, x_n}$ is $\Gamma$-consistent. 
If not, then using $(\structural)$, $\Gamma\proves [x_1,\ldots, x_n]$.
So $\Gamma$ is inconsistent, contradicting our assumption.
  
 Since $\set{x_1,\ldots, x_n}$ is $\Gamma$-consistent, we can extend it to a maximal $\Gamma$-consistent set $\SS$, and by the Truth Lemma, $\SS\in \bigcap_{i=1}^n \semantics{x_i}$. 
  \end{proof}
 
 \begin{theorem}[Completeness] For any sentence $\varphi$, if $\Gamma\models \varphi$, then $\Gamma\proves \varphi$. 
 \label{theorem-completeness-sequent}
 \end{theorem}
 
 \begin{proof}
Suppose $\varphi = [x_1,\dots,x_n]$. If $\Gamma$ is inconsistent, then by (\efq), $\Gamma\proves \varphi$, and we are done. So we may assume that $\Gamma$ is consistent. Assume for contradiction that $\Gamma\not\proves [x_1, \ldots, x_n]$. Consider the canonical model $\Model(\Gamma)$. By Lemma~\ref{lemma-can-works}, $\Model(\Gamma) \models\Gamma$, so $\Model(\Gamma)\models \phi$. By (\structural), the set $\set{x_1,\ldots, x_n}$ is $\Gamma$-consistent. So we can extend it to a maximal $\Gamma$-consistent set $\SS$.   By the Truth Lemma, $\SS\in \bigcap_{i=1}^n \semantics{x_i}$, so in $\Model(\Gamma)\not\models [x_1,\dots, x_n]$, which is a contradiction. 

Now suppose $\varphi = \pair{x_1,\dots,x_n}$. If $\Gamma\models \varphi$, then $\Gamma\cup \set{[x_1,\dots,x_n]}$ has no models, so by Lemma~\ref{lemma-can-works}, $\Gamma\cup \set{[x_1,\dots,x_n]}$ is inconsistent. This means that $\Gamma\cup \set{[x_1,\dots,x_n]}$ proves both $\pair{y_1,\dots,y_m}$ and $[y_1,\dots,y_m]$, so by (\raa), $\Gamma\proves \pair{x_1,\dots,x_n}$, as was to be shown.
  \end{proof}

We have just seen the completeness theorem for  $\langfoursecond^+$.  
Let $\langfour^+$ be the generalization of $\langfour$ obtained by adding sentences of the form $[x_1,\dots,x_n]$ (but not $\pair{x_1,\dots,x_n}$). Then we can restrict our logical system for $\langfoursecond^+$ to all the rules except for (\efq) and (\raa). These rules are still sound for $\langfour^+$, and we will show that they are complete as well.

\begin{corollary}\label{cor:langfourcompleteness}
Let $\proves_0$ be the proof system consisting of the rules $(\axiom)$, $(\structural)$, $(\res)$, and $(\rel)$. Then $\proves_0$ is sound and complete for $\langfour^+$.
\end{corollary}

\begin{proof}
We have already observed the soundness. So suppose $\Gamma$ is a theory in $\langfour^+$ and $\varphi$ is a sentence in $\langfour^+$ such that $\Gamma\models \varphi$. Moving up to the larger language $\langfoursecond^+$, we have $\Gamma\proves\varphi$ by Theorem~\ref{theorem-completeness-sequent}. Our goal is to show that $\Gamma\proves_0\phi$.

We first claim that if $\Gamma$ is any theory in $\langfour^+$ and $\psi$ is any sentence in $\langfoursecond^+$ such that $\Gamma\proves \psi$, then (\raa) is not used in the proof. The argument is by induction on height of the proof tree. Suppose for contradiction that (\raa) is used. We may assume that the root of the proof tree is an application of (\raa):
\[
\infer[\raa]{\pair{x_1,\ldots, x_n}}{\infer*{\pair{y_1,\ldots, y_m}}{\xcancel{[x_1,\ldots, x_n]}}  & \infer*{[y_1,\ldots,y_m]}{\xcancel{[x_1,\ldots, x_n]}} }
\]
The left subtree is a proof of $\pair{y_1,\dots,y_m}$ from $\Gamma\cup \set{[x_1,\dots,x_n]}$. Since $\Gamma\cup \set{[x_1,\dots,x_n]}$ is a theory in $\langfour^+$, by induction (\raa) is not used in this proof. But none of the other rules produce consequences of the form $\pair{y_1,\dots,y_m}$, so this is a contradiction. 

Now it is easy to see that if $\Gamma$ is any theory in $\langfour^+$, then no $\proves$ proof from $\Gamma$ uses (\efq), since none of our rules other than $(\raa)$ allow us to produce or introduce a premise of the form  $\pair{y_1,\dots,y_m}$.

Therefore, if $\Gamma\proves \phi$, then $\Gamma\proves_0 \phi$. 
\end{proof}

\subsection{Open problems concerning $\langfour$ and $\langfoursecond$}

We began this section with the result that the consequence relation for $\langfour$ is $\CONP$ hard.
This implies that, assuming {\sc P $\neq$ NP}, there is no boundedly complete syllogistic proof system for either $\langfour$ or $\langfoursecond$ 
Instead, we added to the syntax and formulated a proof system which went beyond the 
``purely syllogistic''; in that it used schematic rules and also (\raa).   But we did not find proof systems of any kind for the
original languages $\langfour$ and $\langfoursecond$.  We leave this as an open problem.  
(There is a result of possible relevance in~\cite{logic:mossverbs2010}: a syllogistic system for sentences
of the form $(\all{p}{x})$, where $p$ is  a (complemented) noun,
and $x$ is either  a (complemented) noun or a term $(\allterm{r}{q})$, where $q$ is a (complemented) noun.)
 For that matter, 
we also leave open the question of determining the exact
complexities of the consequence relations for $\langfour$ and $\langfoursecond$.

\section{$\langfive$ and $\langfivesecond$: Putting it all together}
\label{section-five}

The largest logics in this paper are $\langfive$ and $\langfivesecond$, as described in Figure~\ref{language-chart}.  We have much less to say about them than about their sub-languages because 
a notational variant of  $\langfivesecond$
has already been studied.
This is the language $\mathcal{R}^{*\dag}$ in~\cite{phmoss}.
Here is the syntax of this language.  We begin with a set
 $\bP$ of nouns a set $\bR$ of \emph{verb atoms}.
 A \emph{verb literal} is either a verb $r$ or its complement $\rbar$.
We define \emph{terms} and \emph{sentences}
via the syntax below:
\begin{equation*}
\begin{split}
\mbox{terms $x$, $y$, $\ldots$} &\qquad p\in \bP \mid \pbar \mid \allterm{r}{p} \mid \allterm{\rbar}{p} 
 \mid \someterm{r}{p} \mid \someterm{\rbar}{p}  \\ 
\mbox{sentences $\phi$, $\psi$, $\ldots$} &\qquad  \all{x}{y} \mid \some{x}{y} 
\end{split}
\end{equation*}
Note that we do \emph{not} have recursion for terms.
In the semantics, we interpret complemented verbs using relational complement:
\[
\semantics{\rbar} = (M\times M)\setminus\semantics{r}.
\]
The rest of the semantics is clear.

\begin{proposition} \label{fivesecond}
There is a translation map $\phi\mapsto\phi^*$ of  sentences in $\mathcal{R}^{*\dag}$ to 
sentences in  $\langfivesecond$ with the following properties:
\begin{enumerate}
\item $\Gamma\models \phi$ iff $\Gamma^*\models\phi^*$, where $\Gamma^* = \set{\psi^* : \psi\in \Gamma}$.
\item  $\phi\mapsto\phi^*$ is computable in $\PTIME$.
\end{enumerate}
\end{proposition}

\begin{proof}
It is sufficient to translate terms $x$ of 
$\mathcal{R}^{*\dag}$.
For this, we use 
\[
\begin{array}{lcl}
(\allterm{\rbar}{x})^* & =&  \notterm{\someterm{r}{x}}\\
(\someterm{\rbar}{x})^* & = & \notterm{\allterm{r}{x}}\\
\end{array}
\]
The point is that the left sides of the equations above have the same interpretations as the right side in every model.
\end{proof}

The translation in the other direction is more complicated due to the complex terms in the languages of this paper.
We use the standard technique of \emph{flattening}.

\begin{proposition} \label{fivesecondsecond} There is a translation map 
$(\Gamma,\phi) \mapsto (\Gamma^*,\phi^*)$ from assertions in  $\langfivesecond$ to 
assertions in  $\mathcal{R}^{*\dag}$ 
 with the following properties:
\begin{enumerate}
\item $\Gamma\models \phi$ iff $\Gamma^*\models\phi^*$.
\item If $\Gamma$ is finite, so is $\Gamma^*$.  In this case, 
$(\Gamma,\phi) \mapsto (\Gamma^*,\phi^*)$ is computable in $\PTIME$.
\end{enumerate}
\end{proposition}

\begin{proof}
Fix $\Gamma$ and $\phi$ in   $\langfivesecond$.
Let $\Terms$ be the set of all terms in $\Gamma\cup\set{\phi}$, including subterms. For each $t\in\Terms$, let $x_t$ be a new noun.
Let $\Delta_1$ be the set of all sentences of $\mathcal{R}^{*\dag}$ below, for $x\in\Terms$:
\[
\begin{array}{l@{\qquad}l}
\all{x_p}{p} & \all{p}{x_p} \\
\all{x_{(\allterm{r}{t})}}{(\allterm{r}{x_t})} &
\all{(\allterm{r}{x_t})}{x_{(\allterm{r}{t})}} 
\\
 \all{x_{(\someterm{r}{t})}}{(\someterm{r}{x_t})} &
 \all{(\someterm{r}{x_t})}{x_{(\someterm{r}{t})}}
 \\
  \all{x_{\pbar}}{\pbar} & \all{\pbar}{x_{\pbar}} \\
 \all{x_{\overline{\allterm{r}{t}}}}{(\someterm{\rbar}{x_t})} &
 \all{(\someterm{\rbar}{x_t})} {x_{\overline{\allterm{r}{t}}}}
\\
 \all{x_{\overline{\someterm{r}{t}}}}{(\allterm{\rbar}{x_t})} &
 \all{(\allterm{\rbar}{x_t})}{x_{\overline{\someterm{r}{t}}}}\\
\all{x_{\notterm{\notterm{t}}}}{x_t} & \all{x_t}{x_{\notterm{\notterm{t}}}}
\end{array}
\]
An easy induction shows that for all $t\in\Terms$, and all models $\Model\models\Delta_1$, 
$\semantics{x_t} = \semantics{t}$.
We translate the sentences of  $\langfivesecond$ to those of $\mathcal{R}^{*\dag}$
(with the new nouns) in the obvious way:  $(\all{t}{u})^*= \all{x_t}{x_u}$, and 
$(\some{t}{u})^*= \some{x_t}{x_u}$.
Let $\Delta_2 = \set{\psi^* : \psi\in \Gamma}$.  Finally, we take $\Gamma^*$ to be $\Delta_1 \cup\Delta_2$,
and $\phi^*$ to be the translation that we just saw.   

We check point (1): $\Gamma\models \phi$ iff $\Gamma^*\models\phi^*$.
Assume that $\Gamma\models\phi$, and let $\Model\models\Gamma^*$.
Due to $\Delta_1$, we have our key fact:
 for all relevant terms $t$, $\semantics{t} = \semantics{x_t}$.
Using this and the fact that $\Model\models\Delta_2$, it follows that $\Model\models\Gamma$.
And so $\Model\models\phi$.   But then using our key fact again,  $\Model\models\phi^*$.
The converse is similar.  
\end{proof}

As shown in~\cite{phmoss}, the consequence relation for
 $\mathcal{R}^{*\dag}$  
 is $\EXPTIME$ complete.  Moreover, there are 
no proof systems which are finite, sound, and complete for the logic, even allowing \emph{reductio ad absurdum}.
(Nevertheless, there are logical systems for $\mathcal{R}^{*\dag}$.  For example,
 Fitch-style system may be found in~\cite{moss:yg,Moss15}.  That system uses individual 
variables, as in first-order logic, but in a controlled way.) 

It follows from the translations in Propositions~\ref{fivesecond} and~\ref{fivesecondsecond} that
the consequence relation for $\langfivesecond$ is  $\EXPTIME$ complete.
Moreover, there can be no sound and complete syllogistic proof system for  $\langfivesecond$,
even allowing all of the  (\casesrule) rules in this paper. Indeed, any rule allowing proof by cases would correspond to a rule in the language of $\mathcal{R}^{*\dag}$  which is derivable from \emph{reductio ad absurdum}. 

We would like to point out that the $\EXPTIME$ hardness result for $\langfivesecond$ extends to the  weaker logic  $\langfive$.
To see this,  we must recall the outline of the argument in~\cite{phmoss}.  The starting point is Spaan's theorem~\cite{spaan}
 that the satisfiability problem for $\lang_U$, modal logic with the universal modality, is 
$\EXPTIME$ hard.   One takes a sentence $\phi$ in $\lang_U$ and translates it to a finite set $S_{\phi}$ of sentences of 
$\mathcal{R}^{*\dag}$ with the property that $\phi$ and $S_{\phi}$ are equisatisfiable.   By our translation, $S_{\phi}$ may be taken to
be a set of sentences in $\langfivesecond$.   By examining the details, $S_{\phi}$ is a  set $S^*_{\phi}$ of sentences in $\langfive$,
together with one additional sentence of the form $(\some{x}{x})$.   The upshot
is that $\phi$ is unsatisfiable iff  $S^*_{\phi}\models \all{x}{\xbar}$.  Note that  $(\all{x}{\xbar})$ is a sentence in $\langfive$.
 In this way, the consequence relation for $\langfive$ is at least as
hard as the (un)satisfiability problem for $\lang_U$.

This paper also explored extensions of syllogistic logic using schemes like (\chains).   It is possible that there is a schematic extension of  $\langfivesecond$, and it is also possible that extensions to the syntax will help.  We have not explored this. 
The logical system for  $\mathcal{R}^{*\dag}$ which uses individual variables adapts to $\langfivesecond$ in a straightforward way, and we expect that
the completeness and finite model properties which were shown in~\cite{moss:yg} hold in the adapted system.

\bibliographystyle{plain}
\bibliography{natlogic}

\begin{thebibliography}{10}

\bibitem{McAllester93}
David~A. McAllester.
\newblock Automatic recognition of tractability in inference relations.
\newblock {\em J. {ACM}}, 40(2):284--303, 1993.

\bibitem{logic:mcA+G92}
David~A. McAllester and Robert Givan.
\newblock Natural language syntax and first-order inference.
\newblock {\em Artificial Intelligence}, 56:1--20, 1992.

\bibitem{logic:moss08}
Lawrence~S. Moss.
\newblock Completeness theorems for syllogistic fragments.
\newblock In F.~Hamm and S.~Kepser, editors, {\em Logics for Linguistic
  Structures}, pages 143--173. Mouton de Gruyter, 2008.

\bibitem{moss:yg}
Lawrence~S. Moss.
\newblock Logics for two fragments beyond the syllogistic boundary.
\newblock In {\em Fields of Logic and Computation: Essays Dedicated to Yuri
  Gurevich on the Occasion of His 70th Birthday}, volume 6300 of {\em LNCS},
  pages 538--563. Springer-Verlag, 2010.

\bibitem{logic:mossverbs2010}
Lawrence~S. Moss.
\newblock Syllogistic logics with verbs.
\newblock {\em Journal of Logic and Computation}, 20(4):947--967, 2010.
\newblock Journal of Logic and Computation.

\bibitem{Moss15}
Lawrence~S. Moss.
\newblock Natural {L}ogic.
\newblock In {\em Handbook of Contemporary Semantic Theory, Second Edition},
  chapter~18. John Wiley \& Sons, 2015.

\bibitem{Moss:LFL}
Lawrence~S. Moss.
\newblock Logic from language.
\newblock ms., Indiana University, to appear.

\bibitem{MossKruckman16}
Lawrence~S. Moss and Alex Kruckman.
\newblock All and only.
\newblock In {\em Partiality, Underspecification, and Natural Language
  Processing}. Cambridge Scholars Publishers, 2017.

\bibitem{prattHartmann2014}
Ian Pratt-Hartmann.
\newblock The relational syllogistic revisited.
\newblock {\em Linguistic Issues in Language Technology}, pages 195--227, 2014.

\bibitem{phmoss}
Ian Pratt-Hartmann and Lawrence~S. Moss.
\newblock Logics for the relational syllogistic.
\newblock {\em Review of Symbolic Logic}, 2(4):647--683, 2009.

\bibitem{Schaefer}
Thomas~J. Schaefer.
\newblock The complexity of satisfiability problems.
\newblock In {\em Proceedings of the Tenth Annual ACM Symposium on Theory of
  Computing}, pages 216--226, San Diego, California, 1978.

\bibitem{spaan}
Edith Spaan.
\newblock {\em Complexity of Modal Logics}.
\newblock PhD thesis, ILLC, University of Amsterdam, 1993.

\end{thebibliography}

\end{document}